\documentclass{article} 
\usepackage{amsmath,amsthm,amsfonts,amssymb,esint,mathtools,extarrows,witharrows,enumitem}

\newcommand{\N}{\mathbb{N}}
\newcommand{\R}{\mathbb{R}}

\makeatletter
\renewcommand*\env@matrix[1][*\c@MaxMatrixCols c]{%
  \hskip -\arraycolsep
  \let\@ifnextchar\new@ifnextchar
  \array{#1}} 
\makeatother

\newtheorem{corollary}{Corollary}[section]
\newtheorem{remark}{Remark}[section]    
\newtheorem{lemma}{Lemma}[section]
\newtheorem{proposition}{Proposition}[section] 
\newtheorem{thm}{Theorem}

\newtheorem*{thm*}{Theorem}
   
\numberwithin{equation}{section}

\newtheorem{propx}{Proposition}

\newtheorem{lemx}{Lemma}

\newtheorem{thmx}{Theorem}

\newcommand\blfootnote[1]{%
  \begingroup
  \renewcommand\thefootnote{}\footnote{#1}%
  \addtocounter{footnote}{-1}%
  \endgroup
}

\title{On the blow-up analysis at collapsing poles for solutions of singular Liouville type equations.}
\author{Gabriella Tarantello}

\begin{document}

\maketitle

\begin{abstract}
We analyse a blow-up sequence of solutions for Liouville type equations
involving Dirac measures with "collapsing" poles.
We consider the case where blow-up occurs exactly at a point where the poles coalesce. 

After proving that a "quantization" property still holds for the "blow-up mass",
we obtain precise pointwise estimates
when blow-up occurs with the least blow-up mass. 

Interestingly, such estimates express the exact analogue of those obtained in 
\cite{Li_Harnack} for solutions of "regular" Liouville equations, 
where the "collapsing" Dirac measures are neglected. 
Such information will be used in \cite{Tar_2} to describe
the asymptotic behaviour of minimizers of the Donaldson functional introduced in 
\cite{Goncalves_Uhlenbeck}, 
yielding to mean curvature 1-immersions of surfaces into hyperbolic 3-manifolds.

\end{abstract}

\section{Introduction}

\blfootnote
{
MSC:
35J75
35J61.\; 
Keywords:
Blow-up analysis, 
Singular Lioville equations,
Blow-up at collapsing poles, 
Mass quantization and pointwise estimate
}

In this note we analyze the behaviour of a sequence of solutions 
for Liouville type equations involving Dirac measures 
whose poles coalesce at a point where the solutions blow-up.
Namely, we have blow-up at a point of "collapsing" singularities. 
This situation naturally arises in the study of Toda systems of Liouville type
in the analysis of the so called "shadow" system discussed in 
\cite{Lee_Lin_Wei_Yang}.
It has motivated the work in \cite{Lin_Tarantello},\cite{Lee_Lin_Tarantello_Yang},
where (after \cite{Ohtusuka_Suzuki}) the phenomenon of "blow-up without concentration"
was recorded and illustrated by various examples.

More conveniently, 
for the following discussion, 
we notice that the role of the poles 
may be replaced by the zeroes of the weight functions 
appearing in the equations
governing the "regular" part of the given solution sequence. 
In this formulation, the issue is to understand 
the blow-up behavior of a (regular) sequence at a point of "collapsing" zeroes.

In fact, this is exactly what may occur 
when studying the asymptotic behavior  of constant mean curvature
(CMC) c-immersions of a closed oriented surface $S$ of genus $\mathfrak{g}\geq 2$
into hyperbolic 3-manifolds. 
Following \cite{Goncalves_Uhlenbeck}, 
it was recently established in \cite{Huang_Lucia_Tarantello_2} 
that for $\vert c \vert <1$, 
the moduli space of all such (CMC) c-immersions can be parametrized 
by elements of the tangent bundle $T(\mathcal{T}_{\mathfrak{g}}(S))$. 
Recall that $\mathcal{T}_{\mathfrak{g}}(S)$ is the space of 
all conformal structures on $S$ identified up to biholomorphisms 
in the same homotopic class of the identity. 
Actually, for $\vert c \vert <1$, the authors in \cite{Huang_Lucia_Tarantello_2}
could label any such (CMC) c-immersion with the   
minimizer (and unique critical point) 
of the Donaldson functional associated to a given element of 
$T(\mathcal{T}_{\mathfrak{g}}(S))$,  
as introduced in \cite{Goncalves_Uhlenbeck}.

In \cite{Tar_2}, we initiated our investigation
about the existence of analogous (CMC) 1-immersions. 
As established in \cite{Tar_2}, 
those immersions can be detected only as "limits", as $c\longrightarrow 1^{-}$,
of the (CMC) c-immersions obtained in \cite{Huang_Lucia_Tarantello_2}. 
In view of the work of Bryant \cite{Bryant} about 
(CMC) 1-immersions into the hyperbolic space $\mathbb{H}^{3}$,
we expect that those (CMC) c-immersions develop  at the limit: 
$c  \longrightarrow 1^{-}$, 
finitely many "punctures", corresponding to possible ramification points. 
We hope to capture such behaviour by means of our blow-up analysis,
with the "punctures" being realized by the blow-up points.
But we face a delicate situation, exactly
when, at the limit, the pull-back metrics of the (CMC) c-immersions  
blow-up at points of "collapsing" zeroes 
of the holomorphic quadratic differentials
identified by the second fundamental form of the immersion.
Indeed, recall that any holomorphic quadratic differential admits 
exactly $4(\mathfrak{g}-1)$ zeroes in $X$,  
counted with multiplicity (see \cite{Griffiths_Harris}).

The pointwise estimates we establish here
(see Theorem \ref{theorem_for_reference})
have helped us to obtain in \cite{Tar_2} the first existence result
about (regular) (CMC) 1-immersions of surfaces
of genus 2 into 3-manifolds of sectional curvature -1,
see  \cite{Tar_2} for details. 

\ 

More precisely,  we start by showing that even in the "collapsing" case, 
the blow-up mass (see \eqref{blow_up_mass_intro} below)
is "quantized" and must take values in $8 \pi \mathbb{N}$. 
This is a somewhat expected property,  
as it has been worked out already in the context of systems in
\cite{Lee_Lin_Wei_Yang},\cite{Lee_Lin_Yang_Zhang}
and in \cite{Lee_Lin_Tarantello_Yang} only for two collapsing zeroes. 
Nevertheless, for the sake of completeness, 
we have chosen to provide in Section \ref{sec_preliminaries} a detailed proof 
(see Theorem \ref{thm_1}), together with other useful extensions of known facts. 
 
More importantly, we obtain sharp pointwise estimates
when blow-up occur with the least blow-up mass $8\pi$. 
Interestingly, such estimates are exactly the analogue
of the estimates obtained by Li in \cite{Li_Harnack}
for solution-sequences of Liouville equations  
with non-vanishing weight functions, 
where none of the "collapsing" issues discussed here arise. 

In order to state our result, 
(after a translation) 
we localize our analysis around the origin. 
Therefore we consider a solution sequence $u_{k}$ satisfying:
\begin{DispWithArrows}<>
&
-\Delta u_{k} = h_{k} e^{u_{k}} -4\pi\sum_{j=1}^{s}\alpha_{j}\delta_{p_{j,k}}
\label{eq_intro}
\\
&
\int B_{r}h_{k}e^{u_{k}}\leq C
\label{volume_control_intro}
\end{DispWithArrows}
with
\begin{equation}\label{distinct_points_collapsing}
s \geq 2 \; \text{ and } \; 
\; \text{ \underline{distinct} points } \; 
p_{j,k}\longrightarrow 0, \; \text{ as } \; k\longrightarrow +\infty, \; j=1,\ldots,s
\end{equation}
\begin{equation}\label{h_k_uniform_bound}
0<a\leq h_{k}\leq b,\; \vert \nabla h_{k} \vert \leq A \; \text{ in } \; B_{r}.
\end{equation}
In addition, we require the following bounded oscillation property:
\begin{equation}\label{bounded_oscillation_intro}
\max_{\partial B_{r}} u_{k} - \min_{\partial B_{r}} u_{k} \leq C
\end{equation}
which is always verified when $u_{k}$ is actually a "localization"
of a globally defined function, for example over a Riemann surface 
(see e.g. Theorem \ref{prop_3.2}). 

We define the "regular" part $\xi_{k}$ of $u_{k}$ as given by
\begin{equation}\label{xi_k_definition_intro}
u_{k}(x)=\xi_{k}(x) + \sum_{j=1}^{n} 2 \alpha_{j} \ln\vert x-p_{j,k} \vert 
\end{equation}
which defines a smooth function in $B_{r}$ and satisfies
\begin{DispWithArrows}<>
&
-\Delta \xi_{k} = (\Pi^{s}_{j=1}\vert x-p_{j,k} \vert^{2 \alpha_{j}})h_{k}e^{\xi_{k}}
\; \text{ in } \; B_{r} 
\label{system_for_xi_equation}
\\
&
\int_{B_{r}} W_{k} e^{\xi_{k}} \leq C
\; \text{ with } \; 
W_{k}(x)=(\Pi^{s}_{j=1}\vert x-p_{j,k} \vert^{2 \alpha_{j}})h_{k}(x) 
\\
& 
\max_{\partial B_{r}} \xi_{k} - \min_{\partial B_{r}} \xi_{k} \leq C
\end{DispWithArrows}
We assume that $\xi_{k}$ admits a blow-up point at the origin, that is
\begin{equation}
\xi_{k}(0)=\max_{\bar{B}_{r}}\xi_{k} \longrightarrow +\infty,  
\; \text{ as } \; k\longrightarrow +\infty.
\end{equation}
Since by our assumptions we know that 
$\xi_{k}$ can admit at most finitely many blow-up points 
(see Remark \ref{sigma_bigger_4_pi}), 
by taking $r>0$ smaller if necessary, 
we can further assume that the origin is the \underline{only} blow-up point 
for $\xi_{k}$ in $B_{r}$, namely:
\begin{equation}\label{only_blow_up_point_for_xi}
\; \forall \; \varepsilon \in (0,r) \; \exists \; C_{\varepsilon}>0 
\; : \; 
\max_{\overline{B}_{r} \setminus B_{\varepsilon} } \xi_{k} \leq C_{\varepsilon}.
\end{equation}
Finally, we define the "blow-up mass" of $\xi_{k}$ at the origin as follows
\begin{equation}\label{blow_up_mass_intro}
\sigma
=
\lim_{\delta \to 0^{+} } 
\underset{ k\to +\infty  }{\underline{\lim}} 
\int_{B_{\delta}}W_{k}e^{\xi_{k}}.  
\end{equation}
The main results  contained in this note can be summarized as follows:
\begin{thm*}
Assume that $\xi_{k}$ satisfies
\eqref{system_for_xi_equation}-\eqref{only_blow_up_point_for_xi}.
If $\alpha_{j}  \in \mathbb{N}$ and 
\eqref{distinct_points_collapsing},\eqref{h_k_uniform_bound} hold then,
for $\sigma$ in \eqref{blow_up_mass_intro} we have: 
\begin{equation*}
\sigma\in 8\pi\mathbb{N}.
\end{equation*}
Furthermore, if $\sigma = 8\pi$ then:
\begin{align}
&
p_{j,k}\neq 0 \; \forall \; j=1,\ldots,s,\;
W_{k}(0)=\Pi^{s}_{j=1}\vert p_{j,k} \vert^{2 \alpha_{j}}h_{k}(0) >0
\; \text{ and } \; 
W_{k}(0)\xrightarrow{k\to \infty}  0^{+}
\notag \\
&
\xi_{k}(x)
=
\ln\frac{e^{\xi_{k}(0)}}{(1+W_{k}(0)e^{\xi_{k}(0)}\vert x \vert^{2})^{2}}
+
O(1)
\; \text{ in } \; B_{r};
\label{shape_of_xi}
\\
&
\int_{B_{r}}\vert \nabla \xi_{k} \vert^{2}
=
16\pi(\,\xi_{k}(0)+\ln(W_{k}(0))\,) + O(1); 
\notag \\
&
\xi_{k}(0)+\min_{\partial B_{r}}\xi_{k}+2\ln W_{k}(0) = O(1)
\notag \\
&
u_{k}(0) + \min_{\partial B_{r}} u_{k}\longrightarrow +\infty, \; \text{ as } \; k\longrightarrow +\infty,
\; \text{ $u_{k}$ in \eqref{xi_k_definition_intro}. } 
\label{u_k_blow_up_at_0}
\end{align}	
\end{thm*}	
By \eqref{u_k_blow_up_at_0}, we see that our original sequence $u_{k}$ 
satisfying \eqref{eq_intro},\eqref{volume_control_intro} and \eqref{bounded_oscillation_intro}
blows-up at the origin if and only if its regular part $\xi_{k}$ 
(see \eqref{xi_k_definition_intro}) 
blows-up there as well. 
Furthermore, as already observed, 
it is remarkable that the pointwise estimate \eqref{shape_of_xi}
is the exact analogue of the one established in \cite{Li_Harnack},
when $W_{k}(0)$ is bounded uniformly away from zero, 
and none of the issues (discussed here) about
"collapsing of zeros" arise.  

Our results are just a first contribution towards the understanding 
of the blow-up phenomena at "collapsing" zeroes  (or poles).
We hope they open the way to the description of "multiple" blow-up profiles,
already present in the "non-collapsing" case, 
according to the analysis
in \cite{Bartolucci_Tarantello_JDE},\cite{Wei_Zhang_1} and \cite{Wei_Zhang_2}.

\section{Blow-up at collapsing zeroes: local analysis.}\label{sec_preliminaries}

Let $\Omega \subset \R^{2}$ be an open, bounded and regular set.  
We consider the sequence: 
$\eta_{k}\in C^{2}(\Omega)\cap C^{0}(\overline{\Omega})$,
satisfying the following Liouville type problem:

\begin{DispWithArrows}<>
& -\Delta \eta_{k} = W_{k} e^{\eta_{k}}   \;\text{in}\;   \Omega   \label{1.a}  \\
& \max_{\partial \Omega}\eta_{k} -\min_{\partial \Omega}\eta_{k} \leq C   
\\
& \int_{\Omega} W_{k}e^{\eta_{k}}\leq C  \label{1.c} 
\end{DispWithArrows}
with a weight function $W_{k}\in L^{\infty}(\Omega)$.

After the work of Brezis-Merle \cite{Brezis_Merle}, 
a vast literature is now available, 
concerning the asymptotic behavior of $\eta_{k}$ (possibly along a subsequence), 
as $k\longrightarrow +\infty$, 
according to various assumptions on $W_{k}$ 
and its vanishing behavior, 
see for example \cite{Bartolucci_Tarantello_Comm_Math_Phys},\cite{Chen_Lin_2},\cite{Chen_Lin_3},\cite{Li_Shafrir},\cite{Ohtusuka_Suzuki},\cite{Tarantello_Book}.

Motivated by our applications in \cite{Tar_2}, 
where we describe the asymptotic behaviour of minimizers of the Donaldson functional
introduced in \cite{Goncalves_Uhlenbeck}
(see \cite{Huang_Lucia_Tarantello_2}), 
here we shall take  $W_{k}$ to satisfy:
\begin{equation}\label{1.1}
W_{k} \geq 0
\; \text{ and } \;  
\Vert W_{k} \Vert_{L^{\infty}(\Omega)}
+
\int_{\Omega} \frac{1}{(W_{k})^{\varepsilon_{0}}}
\leq C,
\; \text{ for some } \; 
\varepsilon_{0} >0.
\end{equation}

A first important information about the sequence $\eta_{k}$ is the following
well known result, stemming from \cite{Brezis_Merle}:

\begin{propx}\label{prop_A}
Assume \eqref{1.a}-\eqref{1.c} and \eqref{1.1}. 
If
$
\overline{\lim}_{k \to +\infty }\int_{\Omega} W_{k}e^{\eta_{k} }
<
4\pi
$, 
then $\eta_{k}^{+} $ is uniformly bounded in $C^{0}_{loc}(\Omega)$.
\end{propx} 	
\begin{proof}
See Proposition 5.3.13 in \cite{Tarantello_Book}.
\end{proof}

Consequently, as in \cite{Brezis_Merle}, 
it is natural to define $z_{0}\in \Omega$ a \underline{blow-up point} 
for the sequence $\eta_{k}$, if the following holds
\begin{equation*}
\; \exists \; z_{k}\longrightarrow z_{0}
\; \text{ with } \; 
\eta_{k}(z_{k}) \longrightarrow +\infty
\; \text{ as } \; 
k\longrightarrow \infty.
\end{equation*}
Moreover, we call the 
\underline{blow-up mass} of $\eta_{k}$ at $z_{0}$ the value:
\begin{equation}\label{1.4}
\sigma(z_{0})
=
\lim_{r \to 0 }
\underline{\lim}_{k\to +\infty} 
\int_{B_{r}(z_{0})} W_{k}e^{\eta_{k} }.			
\end{equation}

\begin{remark}\label{sigma_bigger_4_pi}
By virtue of Proposition \ref{prop_A} we know that 
$\sigma(z_{0})\geq 4\pi$, 
and so, by the assumption \eqref{1.c}, we know that $\eta_{k}$ can admit 
at most a \underline{finite} number of blow-up points.
\end{remark}

From now on we shall denote by $\mathcal{S}$ the finite set (possibly empty) 
of all blow-up points of $\eta_{k}$ in $\Omega$,
and refer to $\mathcal{S}$ as the \underline{blow-up set}.  

\

The following holds.
\begin{propx}\label{prop_B}
Under the assumptions of Proposition \ref{prop_A}, 
$\eta_{k}^{+}$ is uniformly bounded in $C^{0}_{loc}(\Omega \setminus \mathcal{S})$, along a subsequence.
\end{propx}	
\begin{proof}
See Proposition 5.3.17 in \cite{Tarantello_Book}.
\end{proof}

The information above allows us to provide the following "rough" description 
about the asymptotic behaviour of (a subsequence of) $\eta_{k}$. 
It is a general version
of analogous statements available in literature 
(under stronger  assumptions on the vanishing properties of $W_{k}$) 
starting with
\cite{Brezis_Merle},\cite{Li_Shafrir},\cite{Bartolucci_Tarantello_Comm_Math_Phys} 
and then \cite{Ohtusuka_Suzuki},\cite{Lin_Tarantello},\cite{Lee_Lin_Tarantello_Yang}.
See also 
\cite{Lee_Lin_Wei_Yang},\cite{Lee_Lin_Yang_Zhang},\cite{Lin_Wei_Yang_Zhang},\cite{Lin_Wei_Zhang},\cite{Lin_Wei_Zhao_1},\cite{Lin_Wei_Zhao_2} 
for analogous results in the context of Liouville type systems.

\begin{proposition}\label{prop_0.3}
Let $\eta_{k}$ satisfy \eqref{1.a}-\eqref{1.c} with 
$W_{k} \longrightarrow  W$ uniformly in 
$C^{0}_{loc}(\Omega)$, and assume that \eqref{1.1} holds.
Then (along a subsequence) $\eta_{k}$ satisfies one of the following alternatives, as $k\longrightarrow +\infty$:
\begin{enumerate}[label=(\roman*)]
\item $\eta_{k} \longrightarrow -\infty$ uniformly on compact sets of $\Omega$,
\item $\eta_{k} \longrightarrow \eta_{0} $ in $C^{2}_{loc}(\Omega)$, with $\eta_{0}$ satisfying:
$
\begin{cases}
-\Delta \eta_{0} = W e^{\eta_{0}} \; \text{ in } \; \Omega; \notag\\
\int_{\Omega} We^{\eta_{0}} \leq C,  
\end{cases}
$

\item (blow-up): There exists a finite set $\mathcal{S}\neq \emptyset$ 
of blow-up points of $\eta_{k}$ in $\Omega$ such that, 
as $k\longrightarrow +\infty$:
\begin{enumerate} 
\item[a)]
either  (blow-up with "concentration") 
\begin{align}
&
\eta_{k} \longrightarrow -\infty \; \text{ uniformly on compact sets of } \;\Omega \setminus \mathcal{S}
\notag
\\ 
&W_{k}e^{\eta_{k}} \rightharpoonup \sum_{q\in \mathcal{S}}\sigma(q)\delta_{q},
\text{ weakly in the sense of measures,} 
\label{1.6}
\end{align}	
\end{enumerate}
\begin{enumerate}
\item[b)] or (blow-up without concentration)
\begin{align}
&
\eta_{k} \longrightarrow \eta_{0} 
\; \text{ in } \;
C^{2}_{loc}(\Omega \setminus \mathcal{S});
\notag
\\
&
W_{k}e^{\eta_{k}} \rightharpoonup \sum_{q\in \mathcal{S}}\sigma(q)\delta_{q}+We^{\eta_{0}},
\text{ weakly in the sense of measures,} 
\notag
\end{align}	
 and 
\begin{equation*}
\begin{cases}
-\Delta \eta_{0}=We^{\eta_{0}}+\sum_{q\in \mathcal{S}}\sigma(q)\delta_{q}
\; \text{ in } \; \Omega; \\
\int_{\Omega} We^{\eta_{0}}\leq C.
\end{cases}	 
\end{equation*}
\end{enumerate}
Moreover, the blow-up mass $\sigma(q)\geq 4\pi,\; \forall \; q\in \mathcal{S}$.
\end{enumerate} 
\end{proposition}	

\begin{proof}
As already mentioned, the claimed results are available in literature 
under various assumptions on $W_{k}$, and for completeness we highlight here the main 
arguments involved.
Firstly we observe that the sequence:
\begin{equation*}
\phi_{k}=\eta_{k} -\min_{\partial \Omega}\eta_{k} 
\end{equation*}
is uniformly bounded in $\partial \Omega$ and by the Green representation formula
we have:
\begin{equation}\label{1.10} 
\phi_{k}(x) 
=
\frac{1}{2\pi}\int_{\Omega} \ln(\frac{1}{\vert x-y \vert })W_{k}e^{\eta_{k}}dy+\Psi_{k}
\end{equation}
with suitable
$\Psi_{k} \longrightarrow  \Psi$ uniformly in $C^{2}(\Omega)$. 

\
Recalling that $\mathcal{S}$ is the (possibly empty) blow-up set of 
$\eta_{k} $ in $\Omega$, by Proposition \ref{prop_B} we know that $\eta_{k}^{+}$ is uniformly bounded on compact sets of $\Omega \setminus \mathcal{S}$. 
Therefore, we can use well known potential estimates to conclude that (along a subsequence):
\begin{equation}\label{1.11}
\phi_{k} \longrightarrow  \phi_{0}
\; \text{ uniformly in } \; 
C^{1,\alpha}_{loc}(\Omega \setminus \mathcal{S}).
\end{equation}
Next, letting 
$
s_{0}=\frac{\varepsilon_{0}}{1+\varepsilon_{0}}\in (0,1)
$
with $\varepsilon_{0}>0$ in \eqref{1.1}, we check:
\begin{equation*}
\begin{split}
\min_{\partial \Omega}\eta_{k} 
= &
\min_{\Omega}\eta_{k} 
\leq
\fint_{\Omega} \eta_{k}^{+}
=
\frac{1}{s_{0}}\fint s_{0}\eta_{k}^{+}
\leq
C\int_{\Omega} 	e^{s_{0}\eta_{k}} \\
\leq &
C(\int_{\Omega} W_{k}e^{\eta_{k}})^{s_{0}}
(\int_{\Omega} \frac{1}{W_{k}^{\varepsilon_{0}}})^{1-s_{0}}
\leq C.
\end{split}	
\end{equation*}
So (along a subsequence) 
\begin{equation}\label{two_alternatives}
\; \text{ either  } \; 
\min_{\partial \Omega}\eta_{k} \longrightarrow -\infty
\; \text{ on compact sets of } \; 
\Omega \setminus \mathcal{S}
\; \text{ or } \;
\vert \min_{\partial \Omega} \eta_{k}\vert \leq C  .
\end{equation}
If the first alternative holds in \eqref{two_alternatives}, 
then (along a subsequence) we deduce alternative (i) in case 
$\mathcal{S}=\emptyset$, or in case
$\mathcal{S}\neq \emptyset$ then \textit{blow-up with concentration} in alternative (iii). 

On the contrary in case the second alternative hols in \eqref{two_alternatives}, 
then by virtue of \eqref{1.11}, we find (along a subsequence) that
\begin{equation*}
\eta_{k} \longrightarrow \eta_{0} 
\; \text{ uniformly in } \;  
C^{0}_{loc}(\Omega \setminus \mathcal{S}). 
\end{equation*}
Hence we conclude that either (ii) or alternative b) of (iii) hold 
according to whether $\mathcal{S}=\emptyset$ or $\mathcal{S}\neq \emptyset$, 
and in view of Remark \ref{sigma_bigger_4_pi} the proof is completed.  
\end{proof}

As discussed in \cite{Lin_Tarantello} and \cite{Lee_Lin_Tarantello_Yang}, 
all the alternatives of
Proposition \ref{prop_0.3} can actually occur. 
When alternative (iii) holds, then to better understand the behavior of
$\eta_{k}$ around a blow-up point $q\in \mathcal{S}$,
it is crucial to identify the specific value of the blow-up mass 
$\sigma(q)$ in \eqref{1.4}.
To this purpose, for the weight function $W_{k}$, 
we shall work under the following assumption,
\begin{equation}\label{1.12}
\vert \nabla W_{k} \vert \leq A
\; \text{ and } \;
W_{k}
\longrightarrow 
W
\; \text{ in } \;
C^{0}_{loc}(\Omega).  
\end{equation}
According to the results in \cite{Li_Shafrir} and 
\cite{Bartolucci_Tarantello_Comm_Math_Phys} ,
the value of $\sigma(q)$ depends on whether the limiting function
$W$ in \eqref{1.12} vanishes or not at $q \in \mathcal{S}$. 
If locally around $q\in \mathcal{S}$ there holds:
\begin{equation}\label{1.17.a} 
W_{k}(x)=\vert x-p_{k} \vert^{2\alpha}h_{k}(x)
\; \text{ in } \; 
B_{r}(q),\;
\alpha \geq 0,\; 
p_{k}\longrightarrow q 
, 
\end{equation} 
\begin{equation}\label{1.14}
0<a\leq h_{k} \leq b 
\; \text{ and } \; 
\vert \nabla h_{k} \vert \leq A
\; \text{ in } \; B_{r}
\end{equation}
then we know the following:
\begin{thmx}(\cite{Bartolucci_Tarantello_Comm_Math_Phys},\cite{Li_Shafrir}) \label{theorem_2_new}
If $\eta_{k}$ in Proposition \ref{prop_0.3} satisfies alternative (iii) and for 
$q \in \mathcal{S}$ the weight function $W_{k}$ satisfies 
\eqref{1.17.a} and \eqref{1.14} in $B_{r}(q)$, 
then only alternative a) occurs, 
namely we have blow-up with "concentration" as described in \eqref{1.6} and   
\begin{enumerate}[label=(\roman*)]
\item  if $\alpha = 0$ in \eqref{1.17.a} then $\sigma(q)=8\pi$,
\item  if $\alpha>0$ in \eqref{1.17.a} then $\sigma(q)=8\pi(1+\alpha)$. 
\end{enumerate}
\end{thmx}	

\

In this note, we shall focus to the case where,
for $q\in \mathcal{S}$,
we have $W(q)=0$ and $q$ is the accumulation point of different zeroes of $W_{k}$ 
(collapsing of zeroes). 

In view of the applications we have in mind, 
we consider the case where the zeroes of $W_{k}$  have integral multiplicity,
and more precisely  for $q\in \mathcal{S}$ and $r>0$ sufficiently small, 
we assume:
\begin{align}
& W_{k}(x)  
=
\left( 
\Pi^{s}_{j=1}\vert x-p_{j,k} \vert^{2\alpha_{j}}
\right)
h_{k}(x), 
\;
x\in B_{r}(q),
\;
\alpha_{j}\in \mathbb{N}
,\;
s\geq 2
\label{1.15}
\\ 
&
p_{j,k}\neq p_{l,k} \; \text{ for } \; j \neq l
\; \text{ and } \;  
p_{j,k}\longrightarrow q, 
\; \text{ as } \; k \longrightarrow +\infty,
\; \forall \; j=1,\ldots,s.
\label{1.16}
\end{align}
In this case 
we start by proving a \textit{quantization} property for the blow-up mass
in \eqref{1.4},
which completes the result in \cite{Lee_Lin_Tarantello_Yang}  
(where only two zeroes of $W_{k}$ coalesce at $q$)
and follows as in 
\cite{Lee_Lin_Wei_Yang}  and \cite{Lee_Lin_Yang_Zhang},  
where analogous information were deduced in the context of systems.

\begin{thm}\label{thm_1}
Suppose that $\eta_{k}$ in Proposition \ref{prop_0.3} satisfy alternative (iii)
and for $q\in \mathcal{S}$ assume \eqref{1.14}-\eqref{1.16}. 
Then
$
\sigma(q)\in 8\pi \N.
$
\end{thm}

In order to prove Theorem \ref{thm_1} we need some preliminaries. 
First of all, we can "localize" our analysis around the blow-up point $q$.
Indeed by means of \eqref{1.10}, 
we have 
\begin{equation*}
\eta_{k}(x) 
=
\min_{\partial \Omega} \eta_{k}
+
\frac{1}{2\pi}\int_{\Omega} \ln(\frac{1}{\vert x-y \vert })W_{k}e^{\eta_{k}}dy+\Psi_{k}
\end{equation*}
with $\Psi_{k}$ uniformly bounded in $C^{2}(\Omega)$.
Hence we easily check (as in \cite{Bartolucci_Chen_Lin_Tarantello}) 
that $\eta_{k}$ satisfies the bounded oscillation property around $q$.

Thus, after a translation, we can take $q=0$, and for $r>0$ sufficiently small, 
we need to analyze a sequence
$
\xi_{k}\in C^{2}(B_{r})\cap C^{0}(\overline{B}_{r})
$
satisfying: 
\begin{DispWithArrows}<>
& 
-\Delta \xi_{k}
=
(\Pi^{s}_{j=1}\vert x-p_{j,k} \vert^{2\alpha_{j}} ) 
h_{k}(x)e^{\xi_{k}}   
\; \text{ in  } \; 
B_{r}    
\label{2.0} \\
&  
\max_{\partial B_{r}} \xi_{k} 
-
\min_{\partial B_{r}}\xi_{k}  
\leq 
C     
\label{2.1}  \\
& 
\max_{\overline{B}_{r}}\xi_{k}
=\xi_{k}(x_{k})
\longrightarrow  + \infty,
\; \text{ and } \; 
x_{k}\longrightarrow 0,
\; \text{ as } \; 
k\longrightarrow +\infty
\label{2.2},
\end{DispWithArrows}
where 
$s \geq 2,\, \alpha_{j} \in \N$,
$h_{k}$ satisfies \eqref{1.14} in $B_{r}$,
\begin{equation}\label{2.3}
p_{j,k}\neq p_{l,k} \; \text{ for } \; j\neq l,\;
p_{j,k}\longrightarrow 0 
\; \text{ as } \; 
k\longrightarrow +\infty,
\;\forall\; j=1,\ldots,s,
\end{equation}
and
\begin{equation}
\int_{B_{r}} W_{k}e^{\xi_{k}}\leq C,
\; \text{ where } \; 
W_{k}(x)=(\Pi_{j=1}^{s}\vert x-p_{j,k} \vert^{2\alpha_{j}})
h_{k}(x).
\end{equation}
Finally, without loss of generality, in view of \eqref{1.14}, we can assume also that,
\begin{equation}\label{2.6}
h_{k}\longrightarrow h,
\; \text{ in } \; 
C^{0}_{loc}(B_{r})
\; \text{ with } \; 
h(0)=1
\end{equation}
and zero is the \underline{only} blow-up point of $\xi_{k}$ in $B_{r}$,
that is:
\begin{equation}\label{2.7}
\; \forall \; 0 < \delta < r 
\; \exists \; C_{\delta}>0
\; : \; 
\max_{\overline{B}_{r}\setminus B_{\delta}}\xi_{k}
\leq 
C_{\delta}. 
\end{equation}

Clearly, under the above assumptions, we have that \eqref{1.12} holds with 
\begin{equation}\label{W_zero}
W(x)=\vert x \vert^{2\alpha}h(x)
,\;
h(0)=1
\; \text{ and } \;
\alpha=\sum_{j=1}^{s}\alpha_{j} \in \N . 
\end{equation}
Furthermore,  Proposition \ref{prop_0.3} applies to 
$\xi_{k}$ in $B_{r}$,
and by \eqref{2.2} and \eqref{2.7} we know that 
(along a subsequence) $\xi_{k}$ satisfies alternative (iii) with $\mathcal{S}=\{ 0 \} $,
and we have 
\begin{equation}\label{2.8}
m:=
\lim_{\delta \searrow 0  }
\underset{k \to +\infty }{\underline{\lim}}
\frac{1}{2\pi}\int_{B_{\delta}(0)}W_{k}e^{\xi_{k}} \geq 2. 
\end{equation}
Since, as above, by Green representation formula, we can write:
\begin{equation}\label{2.9}
\xi_{k}(x)
=
\min_{\overline{B}_{r}} \xi_{k}
+
\frac{1}{2\pi}
\int_{B_{r}} \ln(\frac{1}{\vert x-y \vert })W_{k}e^{\xi_{k}}dy
+
\psi_{k},
\;
x\in B_{r},
\end{equation}
with $\psi_{k}$ uniformly bounded in $C^{2}(B_{r})$, 
we can use the information in part (iii) of Proposition \ref{prop_0.3}, 
to deduce that (along a subsequence) the following holds:
\begin{equation}\label{2.10}
\nabla \xi_{k}
\longrightarrow -m\frac{x}{\vert x \vert^{2}}
+
\nabla \phi,
\; \text{ uniformly in } \;
C^{1}_{loc}(B_{r}\setminus \{ 0 \} ),
\; \text{ as } \; 
k\longrightarrow +\infty, 
\end{equation}
with a suitable $\phi$ smooth in $B_{r}$.

\

At this point, to establish Theorem \ref{thm_1}, it  suffices to prove the following:
\begin{thm}\label{thm_2.1}
Under the above assumptions, in \eqref{2.8} we have that $m\in 4\N$. 
\end{thm}	
\begin{proof}
We set 
$ 
\alpha=\sum_{j=1}^{s}\alpha_{j}\in \N
$ 
and note that $\alpha\geq 2$. Letting
\begin{equation}\label{2.11b}
\tau_{k}=\max_{j=1,\ldots,s}\vert p_{j,k} \vert 
\longrightarrow 0,
\; \text{ as } \; 
k\longrightarrow +\infty,
\end{equation}
we define 
\begin{equation*}
q_{j,k}=\frac{p_{j,k}}{\tau_{k}}
,\;
j=1,\ldots,s,
\end{equation*}
and since $\vert q_{j,k} \vert \leq 1$, along a subsequence, we can assume:
\begin{equation*}
q_{j,k}\longrightarrow q_{j}
\; \text{ as } \; 
k\longrightarrow +\infty
,\;
j=1,\ldots,s.
\end{equation*}
Therefore, (recalling the normalization $h(0)=1$)
\begin{equation}\label{2.15}
0\leq W_{1,k}(x)
:=
(\Pi_{j=1}^{s}\vert x-q_{j,k} \vert^{2 \alpha_{j}} )h_{k}(\tau_{k}x)
\longrightarrow 
\Pi_{j=1}^{s}\vert x-q_{j} \vert^{2\alpha_{j}}
=
W_{1}(x),
\end{equation}
as $k\longrightarrow +\infty$,
uniformly on compact sets of $\R^{2}$.
We should keep in mind that $W_{1}(x)$ vanishes exactly at the set
\begin{equation}\label{definition_Z_1}
\begin{split}
Z_{1}:= \{ q_{1},\ldots,q_{s} \},
\end{split}
\end{equation}
and the $q_{j}$'s may not be distinct. 
We define
\begin{equation}\label{2.12}
\varphi_{k}(x)
=
\xi_{k}(\tau_{k}x)
+
2(\alpha+1)\ln \tau_{k}
\; \text{ in } \; 
D_{k}=\{ \vert x \vert \leq \frac{r}{\tau_{k}} \} 
\end{equation}
satisfying:
\begin{equation*}
-\Delta \varphi_{k}
=
W_{1,k}(x) e^{\varphi_{k}}
\; \text{ in } \; 
D_{k} 
\end{equation*}
\begin{equation*}
\int_{D_{k}} W_{1,k}e^{\varphi_{k}} \leq C.
\end{equation*}
By scaling \eqref{2.9}, for any given
$x_{1},x_{2} \in \R^{2}$ and $k$ large, we have:
\begin{equation}\label{2.17}
\varphi_{k}(x_{1})-\varphi_{k}(x_{2})
=
\frac{1}{2\pi}
\int_{D_{k}} 
\ln(\frac{\vert x_{2}-y \vert }{\vert x_{1}-y \vert })
W_{1,k}e^{\varphi_{k}}
+
T_{k}(x_{1},x_{2})		
\end{equation}
with $T_{k}$ uniformly bounded in $C^{2}_{loc}$. 
Therefore (as in Lemma 2.2 of \cite{Bartolucci_Chen_Lin_Tarantello}) we find
$R_{0}> 1$ sufficiently large, such that 
$\; \forall \; R\geq R_{0}$ there holds,
\begin{equation*}
\max_{\partial B_{R}}\varphi_{k}
-
\min_{\partial B_{R}}\varphi_{k}
\leq C_{R},
\end{equation*}
with a suitable constant $C_{R}$, depending on $R$ only.

Since $W_{1,k}$ in \eqref{2.15} satisfies \eqref{1.1} and \eqref{1.12}  
in any open and bounded set of $\R^{2}$ 
(for large $k$), 
we can therefore apply Proposition \ref{prop_0.3} to $\varphi_{k}$ 
in any ball $B_{R}$ with $R\geq R_{0}$. 

Furthermore, by a diagonalization process, along a subsequence, we can set 
\begin{equation}\label{2.19}
\mu=\lim_{R \to +\infty } \underset{k \to +\infty }{\underline{\lim} }
\frac{1}{2\pi}\int_{B_{R}} W_{1,k}e^{\varphi_{k}}\leq m,
\end{equation}
and we can describe the asymptotic behavior of $\varphi_{k}$, 
as $k\longrightarrow +\infty$, by one of the following alternatives:
\begin{equation}\label{uniform_on_compact_minus_infinity}
\text{(i)}\quad 
\varphi_{k} \longrightarrow  -\infty
\; \text{ uniformly on compact sets of } \; 
\R^{2}
\; \text{ and } \; 
\mu=0,
\end{equation}
\begin{equation}\label{uniform_on_compact}
\hspace{-60pt}
\text{(ii)} \quad 
\varphi_{k} \longrightarrow  \varphi_{0}
\; \text{ in } \; 
C^{2}_{loc}(\R^{2})
,\;
-\Delta \varphi_{0}=
W_{1}(x)e^{\varphi_{0}}
\; \text{ in } \; \R^{2},
\end{equation}
\begin{equation}\label{2.22}
\hspace{-64pt}
\mu=\frac{1}{2\pi}\int_{\R^{2}}W_{1}(x)e^{\varphi_{0}}
\;\text{($W_{1}$ is given in \eqref{2.15})} \;  
\end{equation}

\

\;(iii) \,
There exists a finite blow-up set
\begin{equation*}
S_{\varphi}
:=
\{ 
q
\; : \; 
\; \exists \; z_{k}\longrightarrow q
\; \text{ and } \; 
\varphi_{k}(z_{k})\longrightarrow +\infty
\} 
\end{equation*}
\quad \quad \quad \;\;\,\, 
such that
\begin{enumerate}[label=(\roman*)]
\item[a)] \underline{either}\;
$\varphi_{k} \longrightarrow  -\infty$
uniformly  on compact sets of 
$\R^{2}\setminus S_{\varphi}$,
\begin{align}
\label{uniform_on_compact_without}
&
W_{1,k}e^{\varphi_{k}}
\longrightarrow 
\sum_{q\in S_{\varphi}}\sigma(q)\delta_{q}
\; \text{ weakly in the sense of measures, } 
\\
&
\label{2.25}
\mu=\frac{1}{2\pi}\sum_{q\in S}\sigma(q)
\; \text{ with } \; 
\sigma(q)\geq 4\pi.
\end{align}	
In addition we know that,
\begin{equation}\label{2.24a}
\begin{split}
& 
\sigma(q)=8\pi
\; \text{ if } \; 
q\neq q_{j}, \; \forall \; j=1,\ldots, s,
\\
&
\sigma(q)=8\pi(1+\alpha_{j})
\; \text{ if } \; 
q=q_{j} \; \text{ and } \; q\neq q_{k}
\; \text{ for } \; 
k\neq j;
\end{split}
\end{equation}
\end{enumerate}

\begin{enumerate}[label=(\roman*)]
\item[b)] 	\underline{or}\;
$\varphi_{k}\longrightarrow \varphi_{0}$ 
in 
$C^{2}_{loc}(\R^{2}\setminus S_{\varphi})$,
\begin{align}
&
\label{2.26}
W_{1,k}e^{\varphi_{k}}
\rightharpoonup 
\sum_{q\in S} \sigma(q)\delta_{q}
+ 
W_{1}
e^{\varphi_{0}}
\; \text{ weakly in the sense of measures,} 
\\
&  
\label{2.27} 
-\Delta \varphi_{0}
=
W_{1}(x)e^{\varphi_{0}}
+
\sum_{q\in S_{\varphi}}\sigma(q)\delta_{q}
\; \text{ in } \; 
\R^{2},\\
& 
\label{2.28}
\mu
=
\frac{1}{2\pi}
(\int_{\R^{2}} W_{1}(x)e^{\varphi_{0}}
+
\sum_{q\in S_{\varphi}}\sigma(q)
).
\end{align}	
Furthermore, when alternative b) holds in (ii)
then necessarily: 
$S_{\varphi}\subset Z_{1}$
(see \eqref{definition_Z_1}), and
blow-up takes place at points in $Z_{1}$ where 
different zero points of $W_{1,k}$ collapse together.
Namely, for $q\in \mathcal{S}_{\varphi}$ there holds:
\begin{equation}\label{2.29a}
\begin{split}
&
q=q_{j_{1}}=q_{j_{2}}=\ldots=q_{j_{m}} \in Z_{1}
\\
& 
\text{with } \; 
m\geq 2,
\; \text{ and } \; 
1\leq j_{1}<\ldots<j_{m} \leq s.
\end{split}
\end{equation}
Moreover, in this case $\sigma(q)$ must satisfy:
\begin{equation}\label{alpha_q}
4\pi 
\leq 
\sigma(q)
<
4\pi(\alpha(q)+1)
\; \text{ for } \;
\alpha(q):=\sum_{k=1}^{m}\alpha(q_{j_{k}}).
\end{equation}
\end{enumerate}
 
Indeed, by \eqref{2.27}, 
we have that 
$\varphi_{0}$ admits a logarithmic singularity at $q$ of order 
$\frac{\sigma(q)}{2\pi}$.  
At the same time we must ensure the integrability of 
$\vert x-q\vert^{2\alpha(q)}e^{\varphi_{0}}$ around $q$
(recall \eqref{2.15} and see \eqref{2.28}),
and this is possible only if
the inequality given in \eqref{alpha_q} holds. 

When either alternative (ii) or b) in (iii) hold, 
then we can argue similarly around infinity.
Therefore, by \eqref{uniform_on_compact}  and \eqref{2.22}, or by 
\eqref{2.27} and \eqref{2.28} respectively, we find that, 
\begin{equation}\label{2.30a}
\varphi_{0}(x)
=
-\mu\ln(\vert x \vert )+O(1),
\; \text{ for } \;  
\vert x \vert \geq 1;
\end{equation}
(see e.g. \cite{Chen_Lin_1}),
and again, by the given integrability condition, we derive that 
\begin{equation}\label{2.30}
\mu=2(\alpha+1+\sigma)
\; \text{ for some } \; 
\sigma>0.
\end{equation}
With the help of the information above  we can establish the following: 
\begin{equation}\label{lem_2.2}
\text{ \underline{Claim 1}: If  } \; \varphi_{k}
\; \text{ satisfies alternative (iii), then } \; 
\sigma(q)\in 8\pi \N,
\; \forall \; q\in S_{\varphi}
\end{equation}

Since $\alpha_{j}\in \N, \; \forall \; j=1,\ldots,s$,
by virtue of \eqref{2.24a}
to establish the claim we only have to consider the case where
$q\in S_{\varphi}$ satisfies \eqref{2.29a}.

To this purpose, 
we go back to the original sequence $\xi_{k}$. 
Actually we replace
$\xi_{k}(x)$ with
$
\xi_{k}(x+\frac{p_{1,k}+p_{2,k}}{2})
$ 
and 
$
\frac{p_{1,k}+p_{2,k}}{2}\longrightarrow 0
$, 
so we can assume, without loss of generality, that,
\begin{equation}\label{2.31}
p_{1,k}=-p_{2,k},\; \forall \;  k\in \N.
\end{equation} 

We proceed by induction on $s$, 
the number of different zeroes of $W_{k}$ collapsing at the origin, 
the only blow-up point of $\xi_{k}$
(see \eqref{2.0},\eqref{2.2},\eqref{2.3},\eqref{2.7}).

\

In case $s=2$, then as shown in \cite{Lee_Lin_Tarantello_Yang}, 
we easily reach the desired conclusion. 
Indeed,  by \eqref{2.11b}  and \eqref{2.31}  we have:
$
\tau_{k}=\vert p_{1,k} \vert =\vert p_{2,k} \vert, 
$
so by letting: 
$q_{k}=\frac{p_{1,k}}{\tau_{k}}$
we have
$\vert q_{k} \vert =1$ and
$\frac{p_{2,k}}{\tau_{k}}=-q_{k}$. 
Thus,  (along a subsequence) we have:
\begin{equation*}
-\Delta \varphi_{k}
=
\vert x-q_{k} \vert^{2\alpha_{1}}\vert x+q_{k} \vert^{2\alpha_{2}}h_{k}(\tau_{k}x)e^{\varphi_{k}}
\; \text{ in } \; 
D_{k} 
\end{equation*}
\begin{equation*}
q_{k}\longrightarrow q_{0}\neq 0,
\; \text{ with } \; 
q_{1}=q_{0} 
\; \text{ and } \; 
q_{2}=-q_{0}.
\end{equation*}
So, for the scaled sequence  $\varphi_{k}$ we do not have to face 
the possibility of a further "collapsing" of zeroes, 
and the desired conclusion follows simply by \eqref{2.24a}. 

\

Next, let $s\geq 3$, and assume that $q\in S_{\varphi}$ satisfies	
\begin{equation*}
q=q_{j_{1}}=q_{j_{2}}=\ldots=q_{j_{m}}
,\;
j_{l}\in \{ 1,\ldots,s \} 
\; \text{ and  } \; m\geq 2.
\end{equation*}
It suffices to show that $m\leq s-1$, 
since then by the induction assumption we could conclude that
$\sigma(q)\in 8\pi \N$, as desired. As above, for suitable $q_{0}\in \R^{2}$, we have:
$
q_{1}=q_{0}
\; \text{ and } \; 
q_{2}=-q_{0}.
$
In case $q_{0}\neq 0$, we see that,
if $q\neq \pm q_{0}$, then
$q\in \{ q_{3},\ldots,q_{s} \}$ 
and necessarily 
$m\leq s-2$. 
While if
$q=q_{0}$ or $q=-q_{0}$,
then
$q\neq q_{2}$ or  $q\neq q_{1}$  respectively, and again 
$m \leq s-1$.
Next we suppose that $q_{0}=0$. In case $q\neq 0$, then we conclude as above that 
$m\leq s-2$. Hence suppose that $q=q_{0}=0$, and let
$j_{0}\in \{ 1,\ldots,s \}$ so that 
(along a subsequence):
$\tau_{k}=\vert p_{j_{0},k} \vert$.
Consequently 
$\vert q_{j_{0}} \vert =1$ and so $q\neq q_{j_{0}}$ and we can still conclude 
that $m\leq s-1$. 
So, if $q$ coincides with $m$-collapsing zeros in $Z_{1}$, then
we have shown that, $2\leq m \leq s-1$, and
as already mentioned, 
the desired conclusion follows by the induction hypothesis.

To proceed further, we recall the following:
\begin{lemx}[Lemma 2.1 of \cite{Lin_Wei_Yang_Zhang}]\label{lem_2.3}
Let $u$ satisfy
\begin{DispWithArrows}<>
& \Delta u + e^{u} =4\pi \sum_{j=1}^{N}\beta_{j}\delta_{q_{j}}   
\; \text{ in } \; \R^{2} \notag \\
& \int_{\R^{2}}e^{u}< +\infty   \notag
\end{DispWithArrows}
with $\beta_{j}\in \N \cup \{ 0 \} $ and $q_{j}\in \R^{2},\; j=1,\ldots,N$. Then
\begin{equation*}
\int_{\R^{2}} e^{u} 
=
4\pi(\sum_{j=1}^{N}\beta_{j}+1+\sigma)\in 8\pi \N
\; \text{ with } \; 
\sigma> 1.
\end{equation*}
\end{lemx}	 

By combining Claim 1 (see \eqref{lem_2.2}) and Lemma \ref{lem_2.3} we obtain:
\begin{equation}\label{lem_2.4}
\text{\underline{Claim 2:}} \quad 
\mu \in 4\N \cup \{ 0 \}.  
\quad\quad\quad\quad\quad\quad\quad\quad\quad\quad
\quad\quad\quad\quad\quad\quad\quad\quad\quad\quad
\end{equation}

To establish \eqref{lem_2.4}, 
notice that in case 
$\varphi_{k}$ satisfies (i) then $\mu=0$. While if $\varphi_{k}$ satisfies (ii), then we can apply Lemma \ref{lem_2.3} to the function,
\begin{equation}\label{2.32}
u:=\varphi_{0}+\sum_{j=1}^{s}2\alpha_{j}\ln\vert x-q_{j} \vert
,\;
\alpha_{j}\in \N
\end{equation}
to conclude, that
\begin{equation*}
\mu=\frac{1}{2\pi}\int_{\R^{2}}\Pi_{j=1}^{s} \vert x-q_{j} \vert^{2\alpha_{j}}e^{\varphi_{0}}
=
\frac{1}{2\pi}\int_{\R^{2}}e^{u}\in 4\N.  
\end{equation*}

On the other hand, if $\varphi_{k}$ satisfies alternative a) of (iii), 
then by \eqref{2.25} and  \eqref{lem_2.2}, still
we find: $\mu \in 4\N$.

Finally, if alternative b) in  (iii) holds, 
then with the notation in \eqref{2.29a},\eqref{alpha_q}
we see that    
$u$ in \eqref{2.32} satisfies:
\begin{equation*}
\Delta u +e^{u}
=
4\pi \sum_{q_{j}\in Z_{1} \setminus S_{\varphi} }\alpha_{j}\delta_{q_{j}}
+
4\pi \sum_{q\in S_{\varphi}}(\alpha(q)-\frac{\sigma(q)}{4\pi})\delta_{q}	
\end{equation*}
and in view of \eqref{alpha_q}, we  can check that Lemma \ref{lem_2.3} applies to $u$. 
As a consequence
\begin{equation*}
\int_{\R^{2}} W_{1}(x)e^{\varphi_{0 }}
=
\int_{\R^{2}} \Pi_{j=1}^{s}\vert x-q_{j} \vert^{2 \alpha_{j}}e^{\varphi_{0 }}
=
\int_{\R^{2}} e^{u}
\in
8\pi \N. 
\end{equation*}
Therefore, by  \eqref{2.28} and \eqref{lem_2.2}, we derive that
$\mu \in 4\N$ as desired.

Notice in particular that, when $\varphi_{k}$ satisfies either 
(ii) or alternative b) of (iii), 
then $\mu$ must satisfy \eqref{2.30} with $\sigma \in \N$. 

Finally to conclude the proof of Theorem \ref{thm_2.1}, 
we recall the following relation between $m$ and $\mu$ established 
e.g. in \cite{Lee_Lin_Wei_Yang}:
\begin{equation}\label{2.33}
m^{2}-\mu^{2} 
= 
4(1+\alpha)(m-\mu).
\end{equation}
For completeness we have included a proof of \eqref{2.33} in the Appendix. 

At this point we see that, if $\varphi_{k}$ satisfies (i) then $\mu=0$ and from 
\eqref{2.33}  we find $m=4(\alpha+1)\in 4\N$, as 
$\alpha=\sum_{j=1}^{s}\alpha_{j}\in \N$. 
In case $\varphi_{k}$ satisfies (ii) then by \eqref{2.30} and \eqref{lem_2.4}, 
$\mu\in 4\N$ and 
$\mu>2(\alpha+1)$. 
Therefore, from \eqref{2.33} we must have that necessarily $m=\mu\in 4\N $.
Finally, if $\varphi_{k}$ satisfies alternative (iii), then again from \eqref{2.33} we have that, either $m=\mu \in 4\N$ or $m=4(\alpha+1)-\mu$, with $\alpha\in \N$ and $\mu\in 4\N$. 
Thus, in any case,
$m\in 4\N$, and the proof is completed.
\end{proof}


The "local" results above can be used to describe the asymptotic behaviour of solutions for Liouville-type equations on a compact Riemann surface
$(X,g)$. Denote by $d_{g}(\cdot,\cdot)$ the distance in $(X,g)$.
We consider the sequence 
$v_{k}\in C^{2,\alpha}(X)$ to satisfy:
\begin{equation}\label{global_problem}
-\Delta v_{k} 
=
R_{k}e^{v_{k}}-f_{k}
\; \text{ in  } \; 
X,
\end{equation}
\begin{equation}\label{n.3.2}
R_{k}(z)
=
\left(
\Pi_{j=1}^{N}(d_{g}(z,z_{j,k}))^{2 \alpha_{j}}
\right)
g_{k}(z),\;
z\in X;
\end{equation}
with
\begin{align}
&
g_{k}\in C^{1}(X)
\; : \; 
a\leq g_{k}\leq b,\;\;
\vert \nabla g_{k} \vert \leq A
\; \text{ and } \; 
g_{k} \longrightarrow g_{0} 
\; \text{ in  } \; 
C^{0}(X);
\\
&
z_{j,k}\in X
\; : \;  
z_{j,k}\neq z_{l,k}
,\;
j\neq l 
\; \text{ and } \; 
z_{j,k}\longrightarrow z_{j}
,\;
j=1,\ldots,N;
\label{n.3.7}
\\
&
f_{k}\in C^{0,\alpha}(X),\;f_{k}\longrightarrow f_{0}
\; \text{ in } \;
L^{p}(X)
\; \text{ for some } \; 
p>1,\;
\int_{X} f_{0}\,dA \neq 0;
\end{align}
and so,
$$
R_{k}(z)
\longrightarrow 
R_{0}(z)
:=
\left(
\Pi_{j=1}^{N}(d_{g}(z,z_{j}))^{2 \alpha_{j}}
\right)
g_{0}(z)
\; \text{ in } \; 
C^{0}(X).
$$
As before we assume that,
\begin{equation}\label{alpha_is_natural}
\alpha_{j}\in \N
,\;
j=1,\ldots,N.
\end{equation}
We denote by
$
Z
=
\{ z \in X\; : \;  R_{0}(z)=0 \}
=
\{ z_{1},\ldots,z_{N} \}
$
the zero set of $R_{0}$, formed by the points 
point $z_{j}$ given in \eqref{n.3.7}, for $j=1,\ldots,N$. 
Again we must keep in mind 
that such points may \underline{not} be distinct,
since,  at the limit, different zeroes of $R_{k}$ could coalesce 
at the same zero of $R_{0}$. 
Therefore, we let $Z_{0}\subset Z$ be the subset 
(possibly empty)  of $Z$, given by 
such "collapsing" zeroes, namely:
\begin{equation*}
\begin{split}
Z_{0} 
=
\{ 
z \in Z
\; : \;
&
 \exists \; s\geq 2,\;
1\leq j_{1}<\ldots<j_{s}\leq N
\; \text{ such that} \; \\ 
& 
z=z_{j_{1}}=\ldots=z_{j_{s}}
\; \text{ and } \;
z \not \in Z \setminus \{  z_{j_{1}},\ldots,z_{j_{s}} \}  
\} . 
\end{split}
\end{equation*}
By combining the "local" results obtained above, we can establish the following:
 
\begin{thm}\label{thm_blow_up_global}
Let $v_{k}$ satisfy \eqref{global_problem} and  
assume \eqref{n.3.2}-\eqref{alpha_is_natural}. 
Then, along a subsequence, one of the following alternatives holds:

\

\noindent
(i) \quad (compactness)\; : \; 
$v_{k}\longrightarrow v_{0}$ in $C^{2}(X)$ with
\begin{equation}\label{n.3.12}
-\Delta v_{0}
=
R_{0}e^{v_{0}}-f_{0}
,\;
\; \text{ in   } \; 
X
\end{equation}

\noindent
(ii)  \quad (blow-up)\; : \;There exists a \underline{finite} blow-up set 
$$
\mathcal{S}
=
\{ q\in X
\; : \; 
\; \exists \; 
q_{k}\longrightarrow q
\; \text{ and } \; 
v_{k}(q_{k})
\longrightarrow 
+ \infty,
\; \text{ as } \; 
k\longrightarrow +\infty
 \} 
$$
such that, 
$v_{k}^{+}$ 
is uniformly bounded in
$C^{0}_{loc}(X\setminus \mathcal{S})$
and,
as $k\longrightarrow +\infty$, 
\begin{enumerate}[label=(\roman*)]
\item[(a)] either (blow-up with concentration)\;:\;
\begin{align}
&
v_{k}\longrightarrow -\infty
\; \text{ uniformly on compact sets of  } \;
X\setminus \mathcal{S},
\notag
\\
&
R_{k}e^{v_{k}}
\rightharpoonup
\sum_{q\in \mathcal{S}}\sigma(q)\delta_{q}
\; \text{ weakly in the sense of measures,} \; 
\sigma(q)\in 8\pi \N.
\notag
\end{align}	
In particular,
$\int_{X} f_{0}\,dA\in 8\pi \N$ in this case. 

\item[(b)] or (blow-up without concentration)\;:\; 
\begin{align}
&
v_{k}\longrightarrow v_{0}
\; \text{ in } \; C^{2}_{loc}(X\setminus \mathcal{S}),
\notag
\\
& 
R_{k}e^{v_{k}}
\rightharpoonup
R_{0}e^{v_{0}}
+
\sum_{q\in \mathcal{S}}\sigma(q)\delta_{q}
\; \text{ weakly in the sense of measures, }\;
\notag
\\
&
-\Delta v_{0}
=
R_{0}e^{v_{0}}
+
\sum_{q\in \mathcal{S}}\sigma(q)\delta_{q}-f_{0}
\; \text{ in  } \; 
X, \;
\sigma(q)\in 8\pi \N.
\notag
\end{align}	
\end{enumerate}

Furthermore, if alternative b) of (ii) holds then
$
\mathcal{S} \subset Z_{0}
$
and so, in this case,  blow-up occurs only at points 
where different zeroes of $R_{k}$ coalesce at the limit. 
\end{thm}	

\begin{proof}
If $\max_{X}v_{k} \leq C$, 
then the right hand side of \eqref{global_problem} 
is uniformly bounded in $L^{p}(X)$, $p>1$.
Hence, by setting 
$v_{k}=w_{k}+d_{k}$ with $d_{k}=\fint_{X}v_{k}$, 
then by elliptic estimates
(see \cite{Aubin_Book}), we see that $w_{k}$ is uniformly bounded in
$C^{1,\alpha}(X)$, and so along a subsequence, we obtain that, 
$w_{k}\longrightarrow w_{0}$ in $C^{1}(X)$, as 
$k\longrightarrow  +\infty$. 
After integration of \eqref{global_problem}, we have:  
\begin{equation}\label{f_integral_non_zero}
\int_{X} e^{v_{k}}R_{k}
=
\int_{X} f_{k}
\longrightarrow 
\int_{X} f_{0}
,
\; \text{ as } \; 
k\longrightarrow +\infty.
\end{equation} 
Since by assumption:
$\int_{X} f_{0} \neq 0$, then from \eqref{f_integral_non_zero} we see that necessarily, 
$\int_{X} f_{0} > 0$,  and so
$
d_{k} 
\longrightarrow 
\log
(
\frac
{\int_{X} f_{0}}
{\int_{X} R_{0}e^{w_{0}}}
)
=
d_{0}.
$
Consequently, 
$v_{k}\longrightarrow v_{0}=w_{0}+d_{0}$
in
$C^{1}(X)$ 
and $v_{0}$ satisfies \eqref{n.3.12}. Thus alternative (i) holds in this case.

Next assume that, (along a subsequence)
\begin{equation*}
\max_{X}v_{k}\longrightarrow +\infty
,
\; \text{ as } \; 
k\longrightarrow +\infty.
\end{equation*}
This implies that the blow-up set
of (a subsequence of) $v_{k}$ is not empty, that is: 
\begin{equation*}
\mathcal{S}
=
\{ 
q\in X \; \mid \; \; \exists \; q_{k} \in X 
\; : \; q_{k}\longrightarrow q \; \text{ and } \; v_{k}(q_{k})\longrightarrow +\infty
\} 
\neq \emptyset.
\end{equation*}

At this point, around any $q\in X$,
we introduce local conformal coordinates centered at the origin,
and (with abuse of notation) we shall denote in the same way the local expressions of 
$v_{k}$ and $f_{k}$ in such coordinates, as defined in a small ball
$B_{r}$.
In this way, in $B_{r}$ we may consider the sequence $\theta_{k}$ satisfying: 
\begin{equation*}
-\Delta \theta_{k}  = f_{k} 
\; \text{ in  } \; 
B_{r} 
\; \text{ and } \; 
 \theta_{k} = 0
\; \text{ in  } \; 
\partial B_{r}. 
\end{equation*}
Hence, $\theta_{k}$ is uniformly bounded in $C^{1,\alpha}(X)$ and (along a subsequence) 
we may assume that,
$\theta_{k}
\longrightarrow 
\varphi_{0}
\; \text{ in  } \; 
C^{1}(X) 
$, as
$k\longrightarrow +\infty$, 
with
$\Delta \theta_{0} =f_{0}$ in $ B_{r} $ and
$\theta_{0} = 0$ in  $\partial B_{r}$.
Therefore, the new sequence  
$\xi_{k}=v_{k}-\theta_{k}$
satisfies:
\begin{equation*}
 - \Delta \xi_{k} = W_{k}e^{\xi_{k}}
\; \text{  in  } \; B_{r} 
\end{equation*}
with suitable $W_{k} \geq 0$ such that,
in $B_{r}$ the following holds:
$$\vert \nabla W_{k} \vert \leq A 
\; \text{ and } \; 
W_{k}\longrightarrow W_{0} \; \text{ uniformly.}$$ 
From \eqref{f_integral_non_zero} also we have:
$$
\int_{B_{r}} W_{k}e^{\xi_{k}} \leq C,
$$
with suitable constant $C>0$. 
Furthermore, 
\begin{equation*}
\begin{split}
&
\; \text{ if } \; 
W_{0}(0)=0
\; \text{ then } \; 
W_{k}(z)=\Pi_{j=1}^{s}\vert z-p_{j,k} \vert^{2\alpha_{j}}h_{k}(z) 
\; \text{ for } \; 
s\geq 2,
\\
& 
\; \text{ and } \; 
h_{k}
\; \text{ satisfies \eqref{1.14} in } \;
B_{r}, 
\;
\alpha_{j} \in \N \text{ and } p_{j,k}\longrightarrow 0,\text{ as } k\longrightarrow \infty. 
\end{split}	
\end{equation*}	

In any case, $W_{k}$ satisfies \eqref{1.1} in $B_{r}$
(with suitable $\varepsilon_{0} >0$) and therefore, 
by Remark \ref{sigma_bigger_4_pi} and \eqref{f_integral_non_zero} 
we conclude that the blow-up set $\mathcal{S}$ is \underline{finite}. 
Also, by using the Green representation formula for $v_{k}$ in $X$ then, 
in the usual way 
(see e.g. \cite{Bartolucci_Chen_Lin_Tarantello} or \cite{Tarantello_Book}),
for every compact set $K\subset X \setminus \mathcal{S}$ 
we can check that,
\begin{equation}\label{min_max_v_k}
\max_{K}v_{k}
-
\min_{K} v_{k}
\leq C,  
\end{equation}
with suitable $C>0$ depending only on $K$. 
 
Thus, by \eqref{min_max_v_k}, 
we can apply alternative (iii) of Proposition \ref{prop_0.3}
and Theorem \ref{thm_1} to obtain (in conformal coordinates) the local blow-up
description of $v_{k}$ around any  blow-up point $q\in \mathcal{S}$. 
At the same time the property \eqref{min_max_v_k} also allows us to patch together 
such local information and arrive at the desired statement in (ii).

\end{proof}

Furthermore, as a direct consequence of Theorem \ref{thm_blow_up_global}, 
we can recover 
a compactness result, well known in the non-collapsing case:
\begin{corollary}
Under the assumptions of Theorem \ref{thm_blow_up_global}, if
$$
\limsup_{k \to +\infty}
\int_{X} R_{k}e^{\xi_{k}}<8\pi
$$
then, along a subsequence, 
$\xi_{k}\longrightarrow \xi_{0}$ in
$C^{1,\alpha}(X)$ with $\xi_{0}$ satisfying \eqref{n.3.12}.   
\end{corollary}

\section{Local estimates in case of least blow-up mass.}

This section is devoted to provide a more detailed description about 
a sequence $\xi_{k}$ 
which 
satisfies the local problem \eqref{2.0} 
and admits
a blow-up point at the origin and \eqref{2.3} holds. 
Namely, $\xi_{k}$ blows-up at a point 
where different zeroes of the weight function $W_{k}$ "collapse" together.

We focus to the case of \underline{least} blow-up mass $8\pi$, namely when
\eqref{2.8} holds as follows: 
\begin{equation}\label{3.0}
m=
\lim_{\delta \searrow 0  }
\underset{k \to +\infty }{\underline{\lim} }
\frac{1}{2\pi}\int_{B_{\delta}}W_{k}e^{\xi_{k}} =4. 
\end{equation}
This is a first important step towards 
the more involved description of "multiple" blow-up profiles,
and we refer to \cite{Lee_Lin_Tarantello_Yang}
for some progress in this direction,  $s=2$ in \eqref{2.3}.

\

With the notation of the previous section, 
a first consequence of $\eqref{3.0}$ is the following:
  
\begin{proposition}\label{prop_3.1}
Let $\xi_{k}$ satisfy \eqref{2.0}-\eqref{2.7} and suppose that \eqref{3.0}  holds.
Then $m=\mu=4$ and the sequence $\varphi_{k}$ in \eqref{2.12} must satisfy 
alternative $a)$ of $(iii)$ in Proposition \ref{prop_0.3}, 
with a \textbf{unique} blow-up point $q_{0}$.   
 
Furthermore if $W_{1}(q_{0})=0$  
(i.e. $q_{0}\in Z_{1}=\{ q_{1},\ldots,q_{s} \} $) 
then $q_{0}$ must satisfy \eqref{2.29a}, 
namely different zeroes of $W_{1,k}$ collapse at $q_{0}$.
\end{proposition}	
\begin{proof}
First of all we can exclude that $\varphi_{k}$ satisfies alternative (i).
Indeed, in this case $\mu=0$, and so by 
\eqref{2.33}  we would have $m=4(1+\alpha)>4$, in contradiction with \eqref{3.0}. 
Therefore,
$0<\mu \leq m$, and since
by \eqref{lem_2.4} we have $\mu\in 4\N$, we see that  necessarily $m=\mu=4$.
This fact allows us to conclude also that $\varphi_{k}$ cannot satisfy 
alternative (ii) or alternative b) of (iii). 
In fact, in this situation, by \eqref{2.30}, we would have:
$\mu=2(\alpha+1+\sigma)$, with $\alpha$ and $\sigma$ positive integers, 
a contradiction to \eqref{3.0}. 
So $\varphi_{k}$ can only satisfy alternative a) of (iii)
with a single blow-up point $q_{0}$. 
Furthermore, if
$q_{0}\in Z_{1}=\{ q_{1},\ldots, q_{s} \} $, 
then \eqref{2.24a} together with \eqref{3.0} allow us to conclude that,
$s\geq 2$ and different zeroes of $W_{1,k}$ must converge to $q_{0}$. 
\end{proof}

To proceed further, we use a different more convenient normalization for $\xi_{k}$
from the previous section. In fact, 
we observe that,  
the translated function:
$\xi_{k}(x+x_{k})$ 
(with the point $x_{k}$ defined in \eqref{2.2})
satisfies an analogous problem, 
possibly in a smaller ball around the origin, 
where properties \eqref{2.0}-\eqref{2.7} continue to hold, 
simply with the point 
$p_{j,k}$ replaced by $(p_{j,k}-x_{k})\longrightarrow 0$, 
as $k\longrightarrow +\infty$. 
More importantly, since 
$x_{k}\longrightarrow 0$ for  $k\longrightarrow +\infty$,
then both
sequences  $\xi_{k}(x)$ and its translated $\xi_{k}(x+x_{k})$
admit the same value of $m$ in \eqref{2.8}. 
Thus, without loss of generality, we can assume that, 
\begin{equation}\label{3.1}
\xi_{k}(0)
=
\max_{\overline{B}_{r}}\xi_{k}
\longrightarrow +\infty,
\; \text{ as } \; 
k\longrightarrow +\infty. 
\end{equation}
Also, after relabeling the indices, 
(along a subsequence)
we may take:
\begin{equation}\label{3.2} 
0
\leq 
\vert p_{1,k} \vert 
\leq
\vert p_{2,k} \vert 
\leq \ldots \leq
\vert p_{s,k} \vert,
\; \text{ with } \; 
s\geq 2;
\end{equation}
so that, 
\begin{equation}\label{3.3}
\tau_{k}=\vert p_{s,k} \vert >0, 
\end{equation}
\begin{equation}
q_{j,k}=\frac{p_{j,k}}{\tau_{k}} \longrightarrow q_{j}, 
\; \text{ as } \; 
k\longrightarrow +\infty,
\; \forall \; j=1,\ldots,s,
\end{equation}
\begin{equation}\label{3.5}
0
\leq 
\vert q_{1} \vert 
\leq
\vert q_{2} \vert 
\leq \ldots \leq
\vert q_{s} \vert 
=1.
\end{equation}
It is important to notice that, with this new normalization, 
we no longer expect  \eqref{2.31} to hold. We have:

\begin{thm}\label{prop_3.2}
Let $\xi_{k}$ satisfy \eqref{2.0},\eqref{2.1},\eqref{2.3}-\eqref{2.7} and assume
\eqref{3.0}-\eqref{3.5}.
Then the points
$p_{j,k} \neq 0,\;1\leq j \leq s$ and 
there exists $s_{1}\in \{ 2,\ldots,s \}$ such that (along a subsequence)
the following holds, 
as $k\longrightarrow +\infty$:
\begin{equation}\label{3.6}
z_{j,k}:=\frac{p_{j,k}}{\vert p_{s_{1},k} \vert }
\longrightarrow 
z_{j}\neq 0,
\; \forall \;
j=1,\ldots,s_{1},
\end{equation} 
and if $s_{1}<s$ then
\begin{equation}\label{3.7}
\begin{split} 
 & \;  
\frac{p_{j,k}}{\vert p_{s,k} \vert }\longrightarrow q_{j} \neq 0
\; \text{ and } \; 
\frac
{\vert p_{j,k} \vert }
{\vert p_{s_{1},k} \vert }
\longrightarrow 
+\infty,
\; \forall \; 
 j=s_{1}+1,\ldots,s. 
\end{split}  
\end{equation}
Moreover
\begin{equation}\label{3.8}
\xi_{k}(0)
+
2\ln\vert p_{s_{1},k} \vert 
+
2 \sum_{j=1}^{s}\alpha_{j}\ln\vert p_{j,k} \vert 
\longrightarrow 
+\infty.
\end{equation}
\end{thm}	

\begin{proof}
If 
\begin{equation}\label{3.9}
q_{j,k}
=
\frac{p_{j,k}}{ \vert p_{s,k} \vert }	
\longrightarrow q_{j}\neq 0,
\; \forall \; j=1,\ldots,s
\end{equation}
then it suffices to choose $s=s_{1}$. Indeed, we only need to check \eqref{3.8}.
To this purpose, by the normalization \eqref{3.1}, we see that
$q_{0}=0$ 
in Proposition \ref{prop_3.1}. So for the sequence
$\varphi_{k}$ in \eqref{2.12} we have
\begin{equation}\label{3.10}
\varphi_{k}(0)
=
\sup_{D_{k}}\varphi_{k}
\longrightarrow 
+\infty.		
\end{equation} 
On the other hand, by \eqref{2.12}, \eqref{3.3} and \eqref{3.9}, we find:
\begin{equation*}
\begin{split}
\varphi_{k}(0)
= &
\;\xi_{k}(0)+2\ln \tau_{k}
+
2\alpha \ln\vert p_{s,k} \vert 	\\
= &
\;\xi_{k}(0)
+
2\ln\vert p_{s,k} \vert 
+
2 \sum_{j=1}^{s}\alpha_{j}\ln\vert p_{j,k} \vert 
+
O(1),
\end{split}
\end{equation*}
and so, \eqref{3.8} follows with $s_{1}=s$ in this case. 
Since $\vert q_{s} \vert =1	$, let us now assume
that there exists 
$\overline{s}\in \{ 1,\ldots, s-1 \} $ such that, 
\begin{equation}\label{3.11}
q_{1}=q_{2}=\ldots=q_{\overline{s}}=0
,
\; \text{ and } \; 
q_{j}\neq 0
\; \text{ for } \; 
j=\overline{s}+1,\ldots, s.
\end{equation}
Since the origin is a blow-up point for $\varphi_{k}$,   
by Proposition \ref{prop_3.1} we know that \eqref{3.11} must hold 
with $\overline{s}\geq 2$.
Clearly, Theorem \ref{thm_1} applies to $\varphi_{k}$ around the origin, 
and it implies that, 
\begin{equation*}
m_{\varphi}
=
\lim_{\delta \searrow 0  }
\lim_{k \to +\infty }
\frac{1}{2\pi}
\int_{B_{\delta}} W_{1,k}e^{\varphi_{k}}
\in 
4\N.
\end{equation*}
Since $m_{\varphi}\leq m=4$, 
we see that necessarily
$m_{\varphi}=4.$ 
This information allows us to apply Proposition \ref{prop_3.1} to $\varphi_{k}$ 
(around the origin), with 
$
0
\leq 
\vert q_{1,k} \vert 
\leq
\vert q_{2,k} \vert 
\leq \ldots \leq
\vert q_{\overline{s},k} \vert 
\longrightarrow 0
,
\; \text{ as } \; 
k\longrightarrow +\infty
$.
Therefore, by setting:  
\begin{equation}\label{tau_1_k}
\tau_{1,k}=\vert q_{\overline{s},k} \vert 
\; \text{ and } \; 
\overline{\alpha}
=
\sum_{j=1}^{\overline{s}}\alpha_{j}\in \N
,\;
\bar{\alpha}\geq 2,
\end{equation}
for the new "scaled" function: \begin{equation*}
\varphi_{1,k}(x)
:=
\varphi_{k}(\tau_{1,k}x)
+
2(\overline{\alpha}+1)\ln\tau_{1,k}
\end{equation*}
we know that (along a subsequence),  $\varphi_{1,k}(0)\longrightarrow \infty$, as $k\longrightarrow +\infty$, 
and actually the origin is the only blow-up point of $\varphi_{1,k}$, 
where "concentration"  occurs, 
as described by alternative a) of (iii) in Proposition \ref{prop_0.3}.

Furthermore (along a subsequence), as $k\longrightarrow +\infty$, we have:
\begin{equation*}
\frac{q_{j,k}}{\tau_{1,k}}
\longrightarrow q_{j}^{(1)}
,	
\; \forall \;
j=1,\ldots, \overline{s}; 
\end{equation*}

with 
$
0\leq \vert q_{1}^{(1)} \vert 
\leq 
\vert q_{2}^{(1)} \vert 
\leq \ldots \leq 
\vert q_{\overline{s}}^{(1)} \vert 
=
1
$. If
\begin{equation}\label{formula_not_to_zero}
\; q_{j}^{(1)}\neq 0,
\; \forall \; j=1,\ldots,\overline{s};
\end{equation}

then from \eqref{3.11} we deduce that, as $k\longrightarrow +\infty$,
\begin{equation}\label{3.15}
\begin{split}
&
\frac
{\vert p_{j,k} \vert }
{\vert p_{\overline{s},k} \vert }
=
\frac
{\vert q_{j,k} \vert }
{\tau_{1,k}}
\longrightarrow 
\vert q_{j}^{(1)} \vert 
\neq 0
,
\; \forall \; j=1,\ldots,\overline{s}, \\
& 
\frac{\vert p_{j,k} \vert }{\vert p_{s,k} \vert }
\longrightarrow \vert q_{j} \vert \neq 0,\; 
\frac
{\vert p_{j,k} \vert}
{\vert p_{\overline{s},k} \vert }
=
\frac
{\vert q_{j,k} \vert }
{\tau_{1,k}}
\longrightarrow +\infty
,
\; \forall \;
j=\overline{s}+1,\ldots, s.
\end{split}
\end{equation}
So, if \eqref{formula_not_to_zero} holds then \eqref{3.6} and \eqref{3.7} 
are verified with $s_{1}=\overline{s}$.

Moreover, to show that also \eqref{3.8} holds with $s_{1}=\bar{s}$,
simply note that 
$\varphi_{1,k}(0)\longrightarrow +\infty$, as $k\longrightarrow +\infty$,
and in view of 
\eqref{2.12},\eqref{tau_1_k},\eqref{3.15} we have 
\begin{equation*}
\begin{split}
\varphi_{1,k}(0)
= \; &
\varphi_{k}(0)+2(\bar{\alpha}+1)\ln \tau_{1,k} \\
= \; &
\xi_{k}(0)
+
2(\alpha+1)\ln\vert p_{k,s} \vert 
+
2(\bar{\alpha}+1)\ln \frac{\vert p_{k,\bar{s}} \vert }{\vert p_{k,s} \vert } \\
= \; &
\xi_{k}(0)+2(\bar{\alpha}+1)\ln\vert p_{k,\bar{s}} \vert
+
2(\alpha-\bar{\alpha})\ln\vert p_{k,s} \vert \\
= \; &
\xi_{k}(0) +2\ln \vert p_{k,\bar{s}} \vert+
\sum_{j=1}^{s}2\alpha_{j}\ln \vert p_{k,j} \vert + O(1).
\end{split}
\end{equation*}
Hence the desired conclusion follows in this case. 

On the contrary, if $q_{j}^{(1)}$ vanishes for some $j\in \{ 1,\ldots,\bar s \} $, 
then we can repeat the argument above for the sequence $\varphi_{1,k}$,
and eventually continue in this way,
by taking further scalings.  
However, 
at each step the number of "collapsing" zeroes decreases, 
and moreover Proposition \ref{prop_3.1} applies
to each of the new scaled sequences.
Therefore, this procedure must stop after finitely many steps. 

In case we assume "by contradiction" that 
$p_{1,k}=0$ identically, then we would end up with a sequence 
blowing-up at zero 
and satisfying a Liouville type problem with a weight function
vanishing at the origin with the same order of 
$\vert x \vert ^{2\alpha_{1}}$.
Thus, to such a sequence,  we can apply \eqref{alpha_q} 
(see \cite{Bartolucci_Tarantello_Comm_Math_Phys}) and obtain a blow-up mass 
$m= \frac{\sigma(0)}{2\pi}=4(1+\alpha_{1})>4$, a contradiction. 
Hence $p_{1,k}\neq 0$,  and so by \eqref{2.3} and \eqref{3.2} we conclude that
$p_{j,k}\neq 0$ for all $1\leq j \leq s$, as claimed. 
Consequently, after finitely many steps, we must arrive at sequence
which blows-up  
at zero and satisfying a Liouville-type problem with a weight function
never vanishing around the origin. 
In other words we have found 
$s_{1}\in \{ 2,\ldots, s \}$ for which \eqref{3.15} holds with 
$\bar{s}$ replaced by $s_{1}$.  
And as before, we check that \eqref{3.8} is also satisfied.
\end{proof}

Theorem \ref{prop_3.2} identifies the appropriate scale to use in order 
to gain good control on $\xi_{k}$ in a tiny neighborhood around the origin.
Indeed, according to Theorem \ref{prop_3.2}, we define:
\begin{equation}\label{4.1}
\varepsilon_{k}
=
\vert p_{s_{1},k} \vert 
\longrightarrow 0
,
\; \text{ as } \; 
k\longrightarrow +\infty
\end{equation}
and let
\begin{equation*}
z_{j,k}
=
\frac{p_{j,k}}{\varepsilon_{k}}
\longrightarrow z_{j}\neq 0
\; \text{ for } \; 
j=1,\ldots,s_{1}.
\end{equation*}
If $s_{1}<s$, we set:
\begin{equation*}
\varepsilon_{j,k}
:=
\frac{\varepsilon_{k}}{\vert p_{j,k} \vert }
\longrightarrow 
0
,
\;
\hat p_{j,k}
:=
\frac{p_{j,k}}{\vert p_{j,k} \vert }
\longrightarrow 
\hat p_{j}
,\;
\vert \hat p_{j} \vert =1
\; \forall \; 
j=s_{1}+1,\ldots,s.
\end{equation*}
We define 
\begin{equation*}
v_{k}(x)
=
\xi_{k}(\varepsilon_{k}x)
+
2\ln \varepsilon_{k}
+
\sum_{j=1}^{s}2\alpha_{j}\ln\vert p_{j,k} \vert,
\end{equation*}
satisfying:
\begin{equation*}
-\Delta v_{k}(x)
=
\Pi_{j=1}^{s_{1}}
\vert x-z_{j,k} \vert^{2\alpha_{j}}
V_{k}(x)e^{v_{k}(x)}
\; \text{ in } \;
D_{k}=B_{\frac{r}{\varepsilon_{k}}} 
\end{equation*}
with
\begin{equation}\label{4.8}
V_{k}(x)
=
\Pi_{j=s_{1}+1}^{s}
\vert \varepsilon_{j,k}x- \hat p_{j,k}  \vert^{2\alpha_{j}}
h_{k}(\varepsilon_{k}x)A_{k}
\end{equation}  
(in case $s_{1}=s$ the product term in \eqref{4.8} is omitted),
and
\begin{equation*} 
A_{k}
=
\Pi_{j=1}^{s_{1}}(\frac{\varepsilon_{k} }{\vert p_{j,k} \vert })^{2 \alpha_{j}}
\longrightarrow 
\Pi_{j=1}^{s_{1}}(\frac{1}{\vert z_{j} \vert })^{2 \alpha_{j}}
=
A_{0} >0
,
\; \text{ as } \; 
k\longrightarrow +\infty.
\end{equation*}
Recall that, $h_{k}$ satisfies \eqref{1.14}
and \eqref{2.6} in $B_{r}$,
and therefore:
\begin{equation}\label{4.10}
R_{k}(x)
:=
\Pi_{j=1}^{s_{1}}\vert x-z_{j,k} \vert^{2 \alpha_{j}}V_{k}(x)
\longrightarrow 
R_{0}(x)
:=
A_{0}\Pi_{j=1}^{s_{1}}\vert x-z_{j} \vert^{2\alpha_{j}}. 
\end{equation}
Clearly, the zero set of $R_{0}$ in \eqref{4.10} is given by:
\begin{equation*}
Z=\{ z_{1},\ldots,z_{s_{1}} \},
\end{equation*}
where the $z_{j}$'s may \underline{not} be necessarily distinct, but they satisfy:
\begin{equation}\label{4.13}
0<\delta_{0}<
\vert z_{1} \vert 
\leq 
\vert z_{2} \vert 
\leq \ldots \leq
\vert z_{s_{1}} \vert 
<
\frac{L_{0}}{4}
\end{equation}
with suitable $\delta_{0}>0$ and $L_{0}\geq 1$.
Also we have:
\begin{equation*}
\int_{D_{k}} R_{k}(x)e^{v_{k}}\leq C.
\end{equation*}
More importantly, in view of
\eqref{3.6},\eqref{3.7},\eqref{3.8}, we can check that,  
\begin{equation*}
v_{k}(0)
=
\max_{\overline{D}_{k}}v_{k}\longrightarrow +\infty,
\; \text{ as } \; 
k\longrightarrow +\infty.
\end{equation*}
By \eqref{3.0}, 
we see that the origin is the \underline{only} blow-up point of $v_{k}$, 
where the well known blow-up analysis of
\cite{Brezis_Merle},\cite{Li_Shafrir},\cite{Bartolucci_Tarantello_Comm_Math_Phys} applies. 
As a consequence we find:
\begin{equation}\label{4.14}
R_{k}e^{v_{k}}\rightharpoonup 8\pi\delta_{0}
\; \text{ weakly in the sense of measures,  } \;
\; \text{ as } \;  
k\longrightarrow +\infty, 
\end{equation}
locally on compact sets. 
Furthermore (see Corollary 5.6.57 of \cite{Tarantello_Book}), 
for every $R>0$
the following well known estimate holds:
\begin{equation}\label{4.15}
\vert 
v_{k}(x)
-
\ln 
\frac
{e^{v_{k}(0)}}
{(1+\frac{e^{v_{k}(0)}}{8}R_{k}(0)\vert x \vert^{2} )^{2}}
\vert 
\leq C_{R}
\; \text{ in  } \;
\overline{B}_{R}, 
\end{equation}
with a suitable constant $C_{R}>0$ depending on $R$ only. 
In particular, from \eqref{4.15} we derive:
\begin{equation}\label{4.16}
\vert v_{k}(x) 
+ 
v_{k}(0)
\vert 
\leq C_{R}
,
\; \forall \; x \in \partial B_{R}.
\end{equation}
\begin{equation}
\max_{\partial B_{R}}v_{k}
-
\min_{\partial B_{R}}v_{k}
\leq C_{R}.
\end{equation}

\begin{lemma}
Let $L_{0}>1$ fixed to satisfy \eqref{4.13}. 
Then for every $R>L_{0}$ there exists $C_{R}$ such that,
\begin{equation}\label{4.18}
\max_{\frac{1}{2}L_{0}\leq \vert y \vert \leq R}
\{  
v_{k}(y)+2(\overline{\alpha}+1)\ln\vert y \vert 
 \}
\leq C_{R}
\end{equation}
with
$\overline{\alpha}=\sum_{j=1}^{s_{1}}\alpha_{j}\in \N$.
\end{lemma}	
\begin{proof}
We start by observing the following:

\

\underline{\textbf{Claim:}} 
\,\; For all
$
\varepsilon >0
$
there exist 
$
k_{\varepsilon}\in \N
\; \text{ and } \; 
r_{\varepsilon}>0
$
such that, for
$ \delta \in (0,r_{\varepsilon})
\; \text{ and } \; 
R\geq \frac{L_{0}}{4}$
we have: 
\begin{equation}\label{4.20}
\underset{B_{\delta}\setminus B_{R \varepsilon_{k}}}{\int}
W_{k}e^{\xi_{k}}
=
\underset{B_{\frac{\delta}{\varepsilon_{k}}}\setminus B_{R}}{\int}
R_{k}e^{v_{k}}
< \varepsilon,
\;\;\; 
\; \forall \; k\geq k_{\varepsilon};
\end{equation}
with $\varepsilon_{k}$ given in \eqref{4.1}. 

\

To establish \eqref{4.20}, recall that by \eqref{3.0} we find
$k_{\varepsilon} \in \N$ and
$r_{\varepsilon}>0$ such that 
\begin{equation*}
\int_{B_{\delta}} W_{k}e^{\xi_{k}}
\leq 8\pi + \frac{\varepsilon}{2}
,
\; \forall \; k\geq k_{\varepsilon}
\; \text{ and } \; 
\delta \in (0,r_{\varepsilon}).
\end{equation*}
On the other hand, by taking $k_{\varepsilon}$ larger if necessary,
from \eqref{4.14}, also we have that
\begin{equation*}
\int_{B_{\frac{L_{0}\varepsilon_{k}}{4}}}W_{k}e^{\xi_{k}}
=
\int_{B_{\frac{L_{0}}{4}}} R_{k}e^{v_{k}}
\geq
8\pi - \frac{\varepsilon}{2}
,
\; \forall \; k\geq k_{\varepsilon}
\end{equation*}
and we immediately derive \eqref{4.20}.

\

To establish \eqref{4.18}, we argue by contradiction and assume there exists $R_{1}>L_{0}$, such that
\begin{equation*}
\; \exists \; 
y_{k}\in B_{R_{1}}\setminus B_{\frac{L_{0}}{2}}
\; : \; 
v_{k}(y_{k})
+
2(\overline{\alpha}+1)\ln\vert y_{k} \vert 
\longrightarrow 
\infty
,
\; \text{ as } \; 
k\longrightarrow +\infty.
\end{equation*}
Define:
\begin{equation*}
\psi_{k}(x)
=
v_{k}(\vert y_{k} \vert x)
+
2(\overline{\alpha}+1)\ln\vert y_{k} \vert
, \; \text{ for } \; 
x\in \Omega
:=
\{  
\frac{1}{2}< \vert x \vert < 2
 \}
\end{equation*}
satisfying:
\begin{equation*}
-\Delta \psi_{k}(x)
=
(\Pi_{j=1}^{s_{1}}\vert x-\frac{z_{j,k}}{\vert y_{k} \vert} \vert^{2\alpha_{j}} )V_{k}(\vert y_{k} \vert x)
e^{\psi_{k}}
\; \text{ in } \; 
\Omega,
\end{equation*}
\begin{equation*}
\psi_{k}(\frac{y_{k}}{\vert y_{k} \vert })
\longrightarrow 
+ \infty
,
\; \text{ as } \; 
k\longrightarrow +\infty.
\end{equation*}
Setting:
\begin{equation*}
R_{1,k}(x)
=
\Pi_{j=1}^{s_{1}}\vert x-\frac{z_{j,k}}{\vert y_{k} \vert } \vert^{2 \alpha_{j}}
V_{k}(\vert y_{k} \vert x),  
\end{equation*}
in view of \eqref{4.13}, we check easily that there exist
$0<a_{1}\leq b_{1}$ and $A>0$ such that
\begin{equation*}
0<a_{1}\leq R_{1,k}(x)\leq b_{1}
\; \text{ and } \; 
\vert \nabla R_{1,k} \vert(x) \leq A
\;\;  
x\in \Omega
,\;
\int_{\Omega } 
R_{1,k}e^{\psi_{k}}
\leq 
C. 
\end{equation*}
Therefore, if along a subsequence, we assume that,
\begin{equation*}
\frac{y_{k}}{\vert y_{k} \vert }
\longrightarrow y_{0},
\; \text{ as } \; 
k\longrightarrow \infty,
\end{equation*}
then $\vert y_{0} \vert =1$, and so 
$y_{0}$ is a blow-up point of $\psi_{k}$ in
$\Omega$.
As above, from 
\cite{Brezis_Merle},\cite{Li_Shafrir},\cite{Bartolucci_Tarantello_Comm_Math_Phys} 
we have (along a subsequence)
\begin{equation*}
R_{1,k}e^{\psi_{k}}
\rightharpoonup  
8\pi \delta_{y_{0}},
\; \text{ weakly in the sense of measures in } \; 
\Omega.
\end{equation*}
and therefore, for $\delta>0$ sufficiently small, there holds
(along a subsequence):
\begin{equation*}
\int_{\{ \vert \, z-\vert y_{k} \vert  y_{0}  \, \vert < \delta \vert y_{k} \vert  \} } 
R_{k}(z)e^{v_{k}(z)}
dz
=
\int_{B_{\delta}(y_{0})} R_{1,k}e^{\psi_{k}}
\geq
4\pi
,
\; \text{ as } \;  
k\longrightarrow +\infty. 
\end{equation*}
Consequently, 
\begin{equation*}
\int_{\{ (1-\delta)\frac{L_{0}}{2} \leq \vert z \vert \leq (1+\delta)R_{1} \} }
R_{k}(z)e^{v_{k}(z)}dz  
\geq 4\pi
,
\; \text{ as } \; 
k\longrightarrow +\infty,
\end{equation*}
a contradiction to \eqref{4.20}.
\end{proof}
 
\ 
 
We can reformulate \eqref{4.18}  in terms of $\xi_{k}$ as follows:
\begin{equation}\label{formula_xi_k_of_x}
\xi_{k}(x)+2(\overline{\alpha}+1)\ln\vert x \vert 
+
\sum_{j=s_{1}+1}^{s}2 \alpha_{j} \ln \vert p_{j,k} \vert 
\leq C_{k}
,\;
\frac{L_{0}}{2}\varepsilon_{k}
\leq 
\vert x \vert 
\leq 
R \varepsilon_{k},
\end{equation}
where the summation term in \eqref{formula_xi_k_of_x} should be dropped in case
$s_{1}=s$.

\noindent
But since for large $k$ we have:
$$
0
\leq
W_{k}(x)
\leq
C_{R}\vert x \vert^{2 \overline{\alpha}}
\Pi_{j=s_{1}+1}^{s}
\vert p_{j,k} \vert^{2\alpha_{j}}
\; \text{ for } \; \frac{L_{0}}{2}\varepsilon_{k}
\leq 
\vert x \vert 
\leq 
R \varepsilon_{k},
$$
then from \eqref{formula_xi_k_of_x} we find:
\begin{equation}\label{4.22}
0
\leq 
W_{k}(x)e^{\xi_{k} }
\leq 
\frac{C_{R}}{\vert x \vert^{2}}
\; \text{ for } \; 
\frac{L_{0}}{2}\varepsilon_{k}
\leq 
\vert x \vert 
\leq
R\varepsilon_{k}
\end{equation}
with suitable $C_{R}>0$ and $k$ large.

\begin{lemma}
For $\varepsilon>0$ sufficiently small there exist 
$k_{\varepsilon}\in \N$ and $C_{\varepsilon}>0$ 
such that, 
\begin{equation}\label{4.23}
\xi_{k}(x)
\leq 
\min_{\partial B_{r}} \xi_{k}
+
(4+\varepsilon) \ln\frac{1}{\vert x \vert }
+
C_{\varepsilon}
,
\; \text{ for } \;
x \in B_{r}\setminus B_{L_{0}\varepsilon_{k}} 
\; \text{ and } \;  
k\geq k_{\varepsilon}.
\end{equation}
\end{lemma}	
\begin{proof}
Firstly, let us fix $\delta \in (0,r)$ sufficiently small, so that
for large $k$ there holds:
\begin{equation}\label{4.24}
M_{k}(\delta)
:=
\frac{1}{2\pi}
\int_{B_{\delta}(0)}W_{k}e^{\xi_{k}}
<
4+\frac{\varepsilon}{2} 
\; \text{ and } \; 
\int_{B_{\delta}\setminus B_{L_{0}\varepsilon_{k} }}
W_{k}e^{\xi_{k}}
<
4\pi \varepsilon.
\end{equation}
Since \eqref{4.23} clearly holds in $B_{r} \setminus B_{\delta}$,
we are left to establish  it in the set
$$
\Omega_{k,\delta}
=
\{ x\; : \; 
R_{0}\varepsilon_{k}
\leq 
\vert x \vert 
\leq 
\delta 
\}. 
$$
In view of \eqref{4.24} we can establish the following:
\begin{equation}\label{4.25}
\; \text{\underline{\textbf{Claim }}\;:\quad   
The inequality \eqref{4.23} holds for } \;  x\; : \; \vert x \vert = L_{0}\varepsilon_{k}.  
\end{equation}
To obtain \eqref{4.25} we use \eqref{2.9} to write:
\begin{equation*}
\begin{split}
\xi_{k}(x)
= &
\min_{\partial B_{r}} \xi_{k}
+
M_{k}(\delta) \ln\frac{1}{\vert x \vert }
+
\frac{1}{2\pi}\int_{\{ \vert x \vert \leq \delta \} } 
\ln(\frac{\vert x \vert }{\vert x-y \vert })W_{k}(y)e^{\xi_{k}(y)} \\
& +
\int_{\{ \delta \leq \vert y \vert \leq r \} } 
\ln(\frac{1}{\vert x-y \vert })W_{k}(y)e^{\xi_{k}(y)}
+
O(1).
\end{split}
\end{equation*}
But, 
for  $\delta \leq \vert y \vert \leq r$ and $\vert x \vert = L_{0} \varepsilon_{k}$,
we see that,
\begin{equation*}
\vert \ln\vert x-y \vert \vert 
\leq 
\vert \ln\vert y \vert  \vert 
+
C,
\; \text{ and } \; 
W_{k}(y)e^{\xi_{k}(y)}
\leq C
\text{(by} \;
\eqref{4.22}),
\end{equation*}
and we deduce:
\begin{equation}\label{4.26}
\begin{split}
\xi_{k}(x)
\leq &
\min_{\partial B_{r}}\xi_{k}
+
M_{k}(\delta)\ln\frac{1}{\vert x \vert }
+
\frac{1}{2\pi}
\underset{\{ \vert y \vert \leq \delta \}}{\int} 
\ln(\frac{\vert x \vert }{\vert x-y \vert })W_{k}(y)e^{\xi_{k}(y)}
+
C.
\end{split}
\end{equation}
In order to estimate the integral term in \eqref{4.26}, 
we let 
\begin{equation*}
D(x) 
= 
\{ \vert y \vert \leq \frac{\vert x \vert }{2} \} 
\cup
\{ 
\vert x-y \vert \geq \frac{\vert x \vert }{2}
\; \text{ and } \; 
\vert y \vert \leq 2\vert x \vert  
\}
\end{equation*}
and observe that, if $y\in D(x)$, then
$\vert \ln\frac{\vert x \vert }{\vert x-y \vert } \vert \leq 4$ and consequently
\begin{equation*}
\vert 
\int_{D(x)} 
\ln(\frac{\vert x \vert }{\vert x-y \vert })
W_{k}(y)e^{\xi_{k}(y)} 
\vert 
\leq C.
\end{equation*}
Moreover, if 
$
2\vert x \vert \leq \vert y \vert \leq \delta
$,
then
$
\frac{\vert x \vert }{\vert x-y \vert }
\leq 
\frac{\vert x \vert }{\vert y \vert }
\frac{1}{\vert \frac{y}{\vert y \vert }-\frac{x}{\vert x \vert } \vert }
\leq 
1
$
and therefore
\begin{equation*}
\int_{\{ 2\vert x \vert \leq \vert y \vert \leq \delta \} } 
\ln(\frac{\vert x \vert }{\vert x-y \vert })W_{k}(y)e^{\xi_{k}(y)}
\leq 
0.
\end{equation*}
So we are left to estimate from above the given integral term on
$B_{\frac{\vert x \vert }{2}}(x)$.

Actually, we can easily check (as above) that, 
for any $\sigma \in (0,\frac{1}{2})$, we can find a suitable constant
$C_{\sigma}>0$ such that
$$
\vert \;
\int_
{
\{ \sigma \vert x \vert \leq \vert y-x \vert \leq \frac{\vert x \vert }{2}  \} 
}
\ln(\frac{\vert x \vert }{\vert x-y \vert })W_{k}(y)e^{\xi_{k}}(y)dy
\; \vert
\leq C_{\sigma}.
$$
Next notice that, if
$y\in B_{\sigma \vert x \vert}(x)$ then
$
(1-\sigma)L_{0}\varepsilon_{k}
\leq \vert y \vert 
\leq 
(1+\sigma)L_{0} \varepsilon_{k}
$,  
and so we can use \eqref{4.22} to estimate
\begin{equation*}
\begin{split}
\int_
{
B_{\sigma \vert x \vert }(x)
} 
\ln(\frac{\vert x \vert }{\vert x-y \vert })W_{k}(y)e^{\xi_{k}(y)}dy
\leq &
C
\int_
{\{ 
\vert \frac{x}{\vert x \vert }-\frac{y}{\vert y \vert } \vert  
\leq \sigma
\}} 
\ln (\frac{1}{\vert \frac{x}{\vert x \vert }-\frac{y}{\vert x \vert } \vert })
\frac{1}{\vert y \vert^{2}}dy \\
= &
C
\int_
{\{
\vert \frac{x}{\vert x \vert }-z \vert 
\leq \sigma  
\}} 
\ln(\frac{1}{\vert \frac{x}{\vert x \vert } - z \vert })
\frac{1}{\vert z \vert^{2}}dz
\leq C.
\end{split}
\end{equation*}
This information together with \eqref{4.24}, implies \eqref{4.25}.

\

To proceed further, we define:
\begin{equation*}
\phi_{k}(x)
=
\xi_{k}(x)
-
\min_{\partial B_{r}}\xi_{k}
-
(4+\varepsilon)\ln\frac{1}{\vert x \vert }
,\;
x\in \Omega_{k,\delta},
\end{equation*}
and in view of \eqref{4.25} we know that, 
$\phi_{k}$ is uniformly bounded from above
on $\partial \Omega_{k,\delta}$, and consequently it satisfies:
\begin{equation*}
\left\{
\begin{matrix*}
\;-\Delta \phi_{k}  = W_{k}e^{\xi_{k}}  &  \;\text{in}\;  &  \Omega_{k,\delta}  \\
\phi_{k}  \leq  C  & \;\text{in}\;  & \partial \Omega_{k,\delta}.  \\
\end{matrix*}
\right.
\end{equation*}
We are going to apply a well known lemma from \cite{Brezis_Merle}
(see e.g. Lemma 5.2.1 of \cite{Tarantello_Book}) to the function
$\tilde{\phi}_{k}$ satisfying:
\begin{equation}\label{3.42_new}
\left\{
\begin{matrix*}[l]
-\Delta \tilde{\phi}_{k}
=
\tilde{f}_{k}  & \; \text{ in  } \;  B_{\delta}& \\
\;\;\;\;\;\,\,\tilde{\phi}_{k} = C & \; \text{ in } \;  \partial B_{\delta }&\\
\end{matrix*}
\right.
\; \text{ with } \; 
\tilde{f}_{k} 
=
\left\{
\begin{matrix*}[l]
W_{k}e^{\xi_{k}} &\; \text{ in } \;  \Omega_{k,\delta} \\
0 &\; \text{ otherwise } 
\end{matrix*}
\right.
\end{equation}
Thus, as a consequence of the second inequality in \eqref{4.24} and 
Lemma 5.2.1 of \cite{Tarantello_Book}, for any $1\leq q <\frac{1}{\varepsilon}$
we find a constant $c_{\varepsilon}=c_{\varepsilon}(q)>0$ such that, 
$
\Vert e^{\tilde{\phi}_{k}} \Vert_{L^{q}(B_{\delta})}
\leq 
c_{\varepsilon}
$. 
Moreover, by the maximum principle, we know that 
\begin{equation}\label{phi_k_estimate}
\phi_{k} \leq \tilde{\phi}_{k} \; \text{ in } \; \Omega_{k,\delta}
\; \text{ and so } \; 
\Vert e^{\phi_{k}} \Vert_{L^{q}(\Omega_{k,\delta})}
\leq c_{\varepsilon}
\end{equation}
for $1\leq q \leq \frac{1}{\varepsilon}$. Since 
$
0\leq W_{k}(x)\leq C\vert x \vert^{2\bar\alpha}
$
in $\Omega_{k,\delta}$ with $\bar\alpha\geq 2$, 
we find that, 
$0
\leq
W_{k}e^{\xi_{k}}
\leq 
C\vert x \vert^{2(\bar \alpha -2)-\varepsilon}e^{\phi_{k}}
$
and, by virtue of \eqref{phi_k_estimate}, we conclude that,
\begin{equation*}
\Vert \tilde{f}_{k} \Vert_{L^{p}(B_{\delta})}
=
\Vert W_{k}e^{\xi_{k}} \Vert_{L^{p}(\Omega_{k,\delta})}
\leq 
C_{\varepsilon}=C_{\varepsilon}(p),
\; \text{ for } \;
1\leq p <\frac{2}{3\varepsilon}. 
\end{equation*}
Therefore, by \eqref{3.42_new}, we can use elliptic estimates to conclude that 
$\tilde{\phi}_{k}$ is uniformly bounded in $B_{\delta}$. In turn, from \eqref{phi_k_estimate}, we deduce:   
\begin{equation*}
\phi_{k}\leq C
\; \text{ in  } \; 
\Omega_{k,\delta}
\end{equation*} 
and \eqref{4.23} is established. 
\end{proof}
The estimate \eqref{4.23} implies in particular that, for any $\varepsilon>0$
sufficiently small, we have:
\begin{equation*}
0
\leq
W_{k}(x)e^{\xi_{k}(x)}
\leq
C\vert x \vert^{2(\bar \alpha - 2)- \varepsilon}
,
\; \text{ for } \; 
L_{0}\varepsilon_{k}\leq \vert x \vert \leq r
\; \text{ and } \;
k\geq k_{\varepsilon} 
\end{equation*}
with $\bar{\alpha} \geq 2$. As a consequence, we have: 
\begin{equation}\label{4.30}
\int_
{
L_{0}\varepsilon_{k}
\leq
\vert x \vert 
\leq 
r
}
\vert \ln\vert x \vert  \,\vert W_{k}(x)e^{\xi_{k}(x)}dx
\leq
C, 
\end{equation}
and we shall take advantage of \eqref{4.30} to refine the estimate
\eqref{4.23} as follows.

\begin{proposition}
We have:
\begin{equation}\label{4.31}
\xi_{k}(x)
=
\min_{\partial B_{r}}\xi_{k}
+
4\ln\frac{1}{\vert x \vert }
+
O(1),
\; \text{ for } \; 
L_{0}\varepsilon_{k}
\leq \vert x \vert 
\leq
r.
\end{equation}
\end{proposition}	
\begin{proof}
As before, we first establish \eqref{4.31}  for 
$\vert x \vert =L_{0}\varepsilon_{k}$.
To this purpose set, 
\begin{equation}\label{mu_k_limit}
\mu_{k}
=
\frac{1}{2\pi}
\underset{\vert y \vert \leq 2L_{0}\varepsilon_{k}}{\int} 
W_{k}(y)e^{\xi_{k}(y)}dy
=
\frac{1}{2\pi}
\underset{\vert z \vert \leq 2L_{0}}{\int}
R_{k}(z)e^{v_{k}(z)}dz
\longrightarrow 4,
\end{equation}
as $k\longrightarrow +\infty$.
Well known estimates established in \cite{Chen_Lin_1} 
(see also\cite{Tarantello_Book})
allow us to conclude that, 
\begin{equation}\label{mu_k_asymptotic}
\vert \mu_{k}-4 \vert v_{k}(0)=O(1).
\end{equation} 
Let 
$x=\varepsilon_{k}x^{\prime}$
with
$\vert x^{\prime } \vert =L_{0}$ and write
\begin{equation}\label{4.32}
\begin{split}
\xi_{k}(x)
= &
\min_{\partial B_{r}} \xi_{k}
+
\mu_{k}\ln\frac{1}{\vert x \vert }
+
\frac{1}{2\pi}
\int_
{
\{\vert y \vert \leq 2L_{0}\varepsilon_{k}\}
} 
\ln(\frac{\vert x \vert }{\vert x-y \vert })W_{k}(y)e^{\xi_{k}(y)}dy \\
& +
\frac{1}{2\pi}
\int_
{
\{ 2L_{0}\varepsilon_{k} \leq \vert y \vert <r  \} 
} 
\ln(\frac{1}{\vert x-y \vert })W_{k}(y)e^{\xi_{k}(y)}dy
+
O(1)
\\ 
= &
\min_{\partial B_{r}}\xi_{k}
+
\mu_{k}\ln \frac{1}{\varepsilon_{k}}
+
\frac{1}{2\pi}
\int_
{
\{\vert y \vert \leq 2L_{0}\}
} 
\ln(\frac{\vert x^{\prime } \vert }{\vert x^{\prime }-y \vert })
R_{k}(y)e^{v_{k}(y)}dy \\
 & +
\frac{1}{2\pi}
\int_
{
\{ 2L_{0}\varepsilon_{k} \leq \vert y \vert <r  \} 
} 
\ln(\frac{1}{\vert y \vert })W_{k}(y)e^{\xi_{k}(y)}dy
+
O(1).
\end{split}
\end{equation}
By virtue of \eqref{4.14} and \eqref{4.15} we find:
$\int_
{
\{\vert y \vert \leq L_{0}\}
} 
\ln(\frac{\vert x^{\prime } \vert }{\vert x^{\prime }-y \vert })R_{k}(y)e^{v_{k}(y)}dy
\longrightarrow 0
$,
 as
$
k\longrightarrow +\infty,
$
while \eqref{4.30} implies that the last integral in \eqref{4.32} is uniformly bounded. 
In conclusion, we have obtained:
\begin{equation}\label{4.33}
\xi_{k}(x)
=
\min_{\partial B_{r}}\xi_{k}
+
\mu_{k}\ln \frac{1}{\varepsilon_{k}}+O(1),
\; \text{ for } \; 
\vert x \vert = \varepsilon_{k}L_{0}.
\end{equation}
As a consequence, for $\vert y \vert =L_{0}$, we have:
\begin{equation*}
\begin{split}
v_{k}(y)
= &\;
\xi_{k}(\varepsilon_{k}y)
+
2\ln \varepsilon_{k}
+
\sum_{j=1}^{s}2 \alpha_{j} \ln\vert p_{j,k} \vert \\
= & \;
\min_{\partial B_{r}} \xi_{k}
+
(2\bar{\alpha}-\mu_{k}) \ln \varepsilon_{k}
+
\sum_{j=s_{1}+1}^{s}2\alpha_{j}\ln\vert p_{j,k} \vert
+
O(1) 
\end{split}
\end{equation*}
At this point, we can use \eqref{4.16} to deduce the following crucial information:
\begin{equation}\label{v_k_in_zero_refined}
v_{k}(0)
=
-\min_{\partial B_{r}}\xi_{k}
+
(2(\bar \alpha +1)-\mu_{k}) \ln\frac{1}{\varepsilon_{k}}
+
\sum_{j=s_{1}+1}^{s}2 \alpha_{j} \ln \frac{1}{\vert p_{j,k} \vert }
+
O(1).
\end{equation}
Since $2(\bar \alpha + 1)-\mu_{k}\longrightarrow 2(\bar \alpha -1) \geq 2$
, as 
$k\longrightarrow +\infty$, from 
\eqref{mu_k_asymptotic} and \eqref{v_k_in_zero_refined} we obtain that,
\begin{equation}\label{v_k_zero_estimate}
0 
<
\ln\frac{1}{\varepsilon_{k}}
\leq C v_{k}(0)
\; \text{ and } \;
\vert \mu_{k}-4 \vert\ln\frac{1}{\varepsilon_{k}}\leq C.  
\end{equation}
for suitable $C>0$. Consequently, by \eqref{4.33},
$$
\vert 
\xi_{k}(x)-4\ln\frac{1}{\varepsilon_{k}}- \min_{\partial B_{r}} \xi_{k}
\vert 
\leq C,
\; \text{ for  } \; 
\vert x \vert =L_{0}\varepsilon_{k}  
$$
and \eqref{4.31} is established for $\vert x \vert = L_{0}\varepsilon_{k}$.

Clearly, by \eqref{2.1}, \eqref{4.31} also holds for $\vert x \vert = r$, 
and we can argue exactly as above for the function 
$
\phi_{k}
=
\xi_{k}(x)-\min_{\partial B_{r}}\xi_{k}
+
4\ln\vert x \vert 
$, 
in order to show that 
$
\Vert \phi_{k} \Vert_
{
L^{\infty}\{ R_{0}\varepsilon_{k}\leq \vert x \vert <r \} 
}
\leq C
$,
and conclude that \eqref{4.31} holds. 
\end{proof}
 
The estimate \eqref{4.31} allows us to show the following additional estimates. 
\begin{proposition}\label{prop_4.4}
Under the above assumptions there holds 
\begin{equation}\label{4.34_former_4.32}
v_{k}(0)
=
-\min_{\partial B_{r}}\xi_{k}
+
2\ln \varepsilon_{k}
-
\sum_{j=1}^{s} 2 \alpha_{j} \ln \vert p_{j,k} \vert 
+
O(1)
\end{equation}
\begin{equation}\label{4.35_former_4.33}
\vert 
v_{k}(y)+v_{k}(0)+4\ln\vert y \vert 
\vert 
\leq 
C
\; \text{ for } \; 
L_{0} \leq \vert y \vert \leq \frac{r}{\varepsilon_{k}}
\end{equation}
\begin{equation}\label{4.36_former_4.34} 
\xi_{k}(0)
=
-\min_{\partial B_{r}}\xi_{k}
-
2 \sum_{j=1}^{s} 2 \alpha_{j} \ln\vert p_{j,k} \vert 
+
O(1)
\end{equation}
\begin{equation}\label{4.37_former_4.35}
\begin{split}
\int_{B_{r}} \vert \nabla \xi_{k} \vert^{2}
= & \; 
16\pi(v_{k}(0)+2\ln\frac{1}{\varepsilon_{k}})
+
O(1)
\\
= & \;
-16 \pi 
(
\min_{\partial B_{r}}\xi_{k}
+
\sum_{j=1}^{s} 2 \alpha_{j} \ln \vert p_{j,k} \vert 
)
+
O(1) \\
= & \;
16 \pi (\xi_{k}(0)+2 \sum_{j=1}^{s}\alpha_{j}\ln \vert p_{j,k} \vert )
+
O(1)
\end{split}
\end{equation}
\end{proposition}	
\begin{proof}
From \eqref{4.31} we derive that,
\begin{equation}\label{4.38_former_4.36} 
\begin{split}
v_{k}(y)
=\; &
\min_{\partial B_{r}}\xi_{k}
-
2\ln \varepsilon_{k}
\\
& \;
+
\sum_{j=1}^{s} 2 \alpha_{j} \ln\vert p_{j,k} \vert 
+
4\ln\frac{1}{\vert y \vert }
+
O(1),
\; \text{ for } \; 
L_{0}\leq \vert y \vert \leq \frac{r}{\varepsilon_{k}}.
\end{split}
\end{equation}
Hence, by using \eqref{4.38_former_4.36} with 
$\vert y \vert = L_{0}\varepsilon_{k}$
together with  \eqref{4.16}, we find:
\begin{equation}\label{4.39_former_4.37}
-v_{k}(0)
=
\min_{\partial B_{r}}\xi_{k}
-
2\ln \varepsilon_{k}
+
\sum_{j=1}^{s} 2 \alpha_{j} \ln\vert p_{j,k} \vert 
+
O(1)
\end{equation}
and \eqref{4.34_former_4.32} is established. 
At this point,  by inserting \eqref{4.39_former_4.37} into \eqref{4.38_former_4.36},
we readily obtain \eqref{4.35_former_4.33}.  
Also we obtain \eqref{4.36_former_4.34} from \eqref{4.34_former_4.32},
 once we recall that,
$
v_{k}(0)
=
\xi_{k}(0)
+
2\ln \varepsilon_{k}
+
\sum_{j=1}^{s} 2 \alpha_{j} \ln\vert p_{j,k} \vert.
$
Finally, to obtain \eqref{4.37_former_4.35}, 
we multiply both sides of the equation \eqref{2.0} by:
$
\xi_{k}-\min_{\partial B_{r}}\xi_{k}\geq 0,
$ 
and then integrate over $B_{r}$. We find
\begin{equation*}
\begin{split}
\int_{B_{r}} \vert \nabla \xi_{k} \vert^{2}
& -
\int_{\partial B_{r}} \frac{\partial \xi_{k}}{\partial \nu}
(\xi_{k}-\min_{\partial B_{r}}\xi_{k})
= 
\int_{B_{r}} W_{k}e^{\xi_{k}}(\xi_{k}-\min_{\partial B_{r}}\xi_{k}) \\
= &
\int_{\{ \vert x \vert \leq L_{0}\varepsilon_{k} \} } 
W_{k}e^{\xi_{k}}(\xi_{k}-\min_{\partial B_{r}}\xi_{k})
+
\int_
{
\{ L_{0}\varepsilon_{k}
\leq \vert x \vert 
\leq r
 \} 
 } 
W_{k}e^{\xi_{k}}(\xi_{k}-\min_{\partial B_{r}}\xi_{k}).
\end{split}
\end{equation*}
Therefore, by means of \eqref{4.31}, we obtain:
\begin{equation*}
\int_
{
\{ L_{0}\varepsilon_{k}
\leq \vert x \vert 
\leq 
r
 \} } 
W_{k}e^{\xi_{k}}(\xi_{k}-\min_{\partial B_{r}}\xi_{k})
\leq
C \int_{B_{r}} \vert x \vert^{2(\bar \alpha - 2)}\ln\frac{1}{\vert x \vert }dx
\leq 
C
\end{equation*}
(recall that $\bar{\alpha} \geq 2$). 
Since $\xi_{k}$ admits the origin as its only blow-up point on $B_{r}$, 
then it is uniformly bounded in $C^{1}$-norm on $\partial B_{r}$,
and in view  \eqref{2.1} we deduce:
$
\vert 
\int_{\partial B_{r}} 
\frac{\partial \xi_{k}}{\partial \nu}
(\xi_{k} - \min_{\partial B_{r}}\xi_{k})
\vert 
\leq C.
$
Finally, we compute:
\begin{equation}\label{big_equation}
\begin{split}
\int_
{
\{ \vert x \vert \leq L_{0}\varepsilon_{k} \}  
} & 
W_{k}(x)
e^{\xi_{k}(x)}
(\xi_{k}(x)-\min_{\partial B_{r}}\xi_{k})dx \\
= &
\underset{\{ \vert x \vert  \leq L_{0}\varepsilon_{k} \} }{\int}
W_{k}(x)(\xi_{k}(x)-\xi_{k}(0))dx
\\
&
+
(\xi_{k}(0)-\min_{\partial B_{r}}\xi_{k})
\,
\int_{\{ \vert x \vert \leq L_{0}\varepsilon_{k}  \} }W_{k}(x)e^{\xi_{k}(x)}dx 
\\
= &
\underset{\{ \vert y \vert \leq L_{0} \} }{\int}
R_{k}(y)e^{v_{k}(y)}
(v_{k}(y)-v_{k}(0)) dy
\\
&
+
(\xi_{k}(0)-\min_{\partial B_{r}}\xi_{k})
\underset{\{ \vert y \vert \leq L_{0} \} }{\int} R_{k}(y)e^{v_{k}(y)}dy \\
= &
\underset{\{ \vert y \vert \leq L_{0} \} }{\int}
R_{k}(y)e^{v_{k}(y)}(v_{k}(y)-v_{k}(0))dy  + \\
& 
-
2
( 
\min_{\partial B_{r}}\xi_{k}
+
\sum_{j=1}^{s} 2 \alpha_{j} \ln\vert p_{j,k} \vert 
)
\underset{\{ \vert y \vert \leq L_{0} \} }{\int}  R_{k}(y)e^{v_{k}(y)}dy
)
+
O(1)
.
\end{split}
\end{equation}
Since by 
\eqref{mu_k_limit} and \eqref{v_k_zero_estimate} we have:
$$
( v_{k}(0)+2\ln\frac{1}{\varepsilon_{k}})
\vert \int_{\{ \vert y \vert \leq L_{0} \} } R_{k}(y)e^{v_{k}(y)}dy-8\pi \vert
\leq C, 
$$
then, by \eqref{4.34_former_4.32}, 
we can estimate the last term in \eqref{big_equation} as follows: 
\begin{equation*}
\begin{split}
-
2
(
\min_{\partial B_{r}}\xi_{k}
& +
\sum_{j=1}^{s} 2 \alpha_{j} \ln\vert p_{j,k} \vert
) 
\int_{\{ \vert y \vert \leq L_{0} \} }  R_{k}(y)e^{v_{k}(y)}dy
) 
\\
= \; &
2(v_{k}(0)+2\ln\frac{1}{\varepsilon_{k}})\int_{\{ \vert y \vert \leq L_{0}  \} }R_{k}(y)e^{v_{k}(y)}dy
\\
= \; &
(v_{k}(0)+2\ln\frac{1}{\varepsilon_{k}})16\pi \\
& +
2(v_{k}(0)+2\ln\frac{1}{\varepsilon_{k}})
(\int_{\{ \vert y \vert \leq L_{0} \} } R_{k}(y)e^{v_{k}(y)}dy-8\pi)
\\
= \; &
(v_{k}(0)+2\ln\frac{1}{\varepsilon_{k}})16\pi
+
O(1).
\end{split}
\end{equation*}
Finally, we use \eqref{4.15} to estimate 
\begin{equation*}
\begin{split} 
\vert  
\int_{\{ \vert x \vert \leq L_{0} \} } 
R_{k}(x)
&
e^{v_{k}(x)}(v_{k}(x)-v_{k}(0))
\vert 
dx \\
\leq  &
C \int_{\{ \vert x \vert \leq L_{0} \} } 
\frac{e^{v_{k}(0)}}{(1+\frac{e^{v_{k}(0)}}{8}R_{k}(0)\vert x \vert^{2})^{2}}
\ln
(
1+\frac{e^{v_{k}(0)}}{8}R_{k}(0)\vert x \vert^{2}
)
dx
\\
\leq  &
C \int_{\R^{2}} \frac{dy}{(1+\vert y \vert^{2})^{2}}\ln(1+\vert y \vert^{2})
\leq 
C
\end{split}
\end{equation*}
and so \eqref{4.37_former_4.35} follows.
\end{proof}

\

Since $R_{k}(0)=h_{k}(0)$, 
we can combine \eqref{4.15} and \eqref{4.37_former_4.35}
to find
\begin{equation}\label{4.40_former_4.38}
\vert 
v_{k}(y)
- 
\ln 
\frac
{e^{v_{k}(0)}}
{(1+\frac{e^{v_{k}(0)}}{8}h_{k}(0)\vert y \vert^{2})^{2}}
\vert 
\leq 
C
,
\; \text{ for } \; 
\vert y \vert \leq \frac{r}{\varepsilon_{k}}.
\end{equation}
and obtain in particular that: \; 
$
\int_{\{ \vert y \vert \leq \frac{r}{\varepsilon_{k}} \} }
e^{v_{k}(y)}dy
\leq C 
$.

\

In view of Theorem \ref{prop_3.2},  we know that 
$p_{j,k}\neq 0 \; \forall \; k\in \N$, and so 
$W_{k}(0)=\Pi_{j=1}^{s}\vert p_{j,k} \vert^{2\alpha_{j}}h_{k}(0) >0 $. 
Hence we can formulate \eqref{4.40_former_4.38} in terms of $\xi_{k}$ as follows:

\begin{corollary}
Under the above assumptions we have,
$0<W_{k}(0)\longrightarrow 0$, as $k\longrightarrow +\infty$, and
\begin{equation}\label{xi_k_uniform_estimate}
\xi_{k}(x)
=
\ln (\frac{e^{\xi_{k}(0)}}{(1+\frac{e^{\xi_{k}(0)}}{8}W_{k}(0)\vert x \vert^{2} )^{2}})
+
O(1)
\; \text{ in } \; 
B_{r}.
\end{equation}
\end{corollary}

\

\

It is interesting to compare  \eqref{xi_k_uniform_estimate} with 
the analogous one "bubble" estimate established in 
\cite{Chen_Lin_1} and \cite{Li_Harnack}
(see e.g. Theorem 0.3 in \cite{Li_Harnack}) concerning the profile of 
blow-up solutions of
\eqref{2.0}-\eqref{2.2}, when $W_{k}\xrightarrow{k\to +\infty} W$ uniformly in $B_{r}$
and $W(0)>0$, 
in which case \eqref{3.0} is automatically satisfied 
(see (i) of Theorem \ref{theorem_2_new}) and no collapsing issues arise.
Indeed,
our pointwise estimate in  \eqref{xi_k_uniform_estimate} 
is the striking exact analogue 
of the one provided in Theorem 0.3 of \cite{Li_Harnack},
carried over to the case where  $W$ satisfies 
\eqref{W_zero} and in particular $W(0)=0$. 

\

Furthermore, for the sequence:
$$
u_{k}(x)
= 
\xi_{k}(x)+\sum_{j=1}^{s} 2 \alpha_{j}
\log\vert x-p_{j,k} \vert, 
$$
satisfying:
\begin{equation}\label{singular_equation_for_u_k}
-\Delta u_{k}
=
h_{k}e^{u_{k}}
- 
4\pi \sum_{j=1}^{s}\alpha_{j}\delta_{p_{j,k}}
\; \text{ in } \; 
B_{r},
\end{equation}
we realize that, the estimate \eqref{4.36_former_4.34}  
stated for $\xi_{k}$ in Proposition \ref{prop_4.4}
reduces, in terms of $u_{k}$, to the following:
\begin{equation*}
u_{k}(0)+\min_{\partial B_{r}}u_{k}
=
-\ln W_{k}(0)\longrightarrow +\infty,
\; \text{ as } \; 
k\longrightarrow +\infty,
\end{equation*}
as stated in \eqref{u_k_blow_up_at_0}, and implying (by \eqref{only_blow_up_point_for_xi}) 
that blow-up for solutions of the singular problem 
\eqref{singular_equation_for_u_k}
is equivalent to blow-up for its "regular" part. 

\

We shall use the estimates established here to describe 
the asymptotic behaviour of minimizers of the Donaldson functional 
considered in \cite{Goncalves_Uhlenbeck},\cite{Huang_Lucia_Tarantello_2},
and to obtain in particular that for Riemann surfaces of genus 2, 
it is always bounded from below, 
although not coercive. 
Furthermore, we shall provide rather precise information on when the
infimum is attained.
In this way we obtain the first existence result about (CMC) 1-immersions
of a closed orientable surface 
of genus 2 into  hyperbolic 3-manifolds.

For this purpose, we conveniently summarize the results established above
for a sequence $\xi_{k}$ satisfying: 
\begin{equation}\label{assumptions_for_theorem_for_references}
\begin{cases}
-\Delta \xi_{k}(x)
=
(\Pi_{j=1}^{s}\vert x-p_{j,k} \vert^{2\alpha_{j}} )h_{k}(x)+g_{k}(x) 
\; \text{ in  } \; B_{r} 
\\
s\geq 2 \; \text{ and  } \; \alpha_{j} \in \N
\; \text{ for  } \; j=1,\ldots,s
\\
\xi_{k}(0)=\max_{B_{r}}\xi_{k} \longrightarrow +\infty, 
\; \text{ as } \; k\longrightarrow +\infty,
\\
 \forall \;\; 0 < \delta < r 
\;\; \exists \; C_{\delta}>0
\; : \; 
\max_{\overline{B}_{r}\setminus B_{\delta}}\xi_{k}
\leq 
C_{\delta}
\\
\max_{\partial B_{r}} \xi_{k} -\min_{\partial B_{r}}\xi_{k} \leq C
\\
\int_{B_{r}} W_{k}e^{\xi_{k}} \leq C
,\;
W_{k}(x):=
(\Pi_{j=1}^{s}\vert x-p_{j,k} \vert^{2\alpha_{j}})h_{k}(x)  
\end{cases}
\end{equation}
\begin{thm}\label{theorem_for_reference}
Suppose $\xi_{k}$ satisfies \eqref{assumptions_for_theorem_for_references}
with the points $p_{j,k}$  
satisfying \eqref{2.3} and \eqref{3.2}, 
$h_{k}$ satisfying \eqref{1.14} and \eqref{2.6}, 
and $g_{k}$ a convergent sequence in $L^{p}(B_{r}),\; p>1$. 

If \eqref{3.0} holds, then the points
$p_{j,k}\neq 0,\;j=1,\ldots,s$;
and there exists 
$s_{1}\in \{ 2,\ldots,s \} $ such that (along a subsequence) 
\begin{equation*}
z_{j,k}:=\frac{p_{j,k}}{\vert p_{s_{1},k} \vert }
\longrightarrow 
z_{j}\neq 0, 
\; \forall \;
j=1,\ldots,s_{1},
\end{equation*} 
\begin{equation*}
\begin{split}  
\; \text{ if } \; s_{1}<s \; \text{ then } \;   
\frac{p_{j,k}}{\vert p_{s,k} \vert }\longrightarrow q_{j} \neq 0
\; \text{ and } \; 
\frac
{\vert p_{j,k} \vert }
{\vert p_{s_{1},k} \vert }
\longrightarrow 
+\infty,
\; \forall \; 
 j=s_{1}+1,\ldots,s. 
\end{split}
\end{equation*}
Moreover, 
\begin{equation*}
\begin{split}
& \xi_{k}(0) \; 
+  
2\ln \vert p_{s_{1},k} \vert   + \ln (W_{k}(0)) \longrightarrow +\infty, 
\; \text{ as } \; k\longrightarrow +\infty,
\\
& \xi_{k}(0) 
+
\left( \min_{\partial B_{r}} \xi_{k} + 2\ln (W_{k}(0))\right) = O(1)
\\
&
\xi_{k}(x)
=  
\ln  
\left(
\frac
{e^{\xi_{k}(0)}}
{(1+ \frac{W_{k}(0)}{8}e^{\xi_{k}(0)}\vert x \vert^{2})^{2} }
\right)
+
O(1)
,\; x \in B_{r} \\
&
\int_{B_{r}}\vert \nabla  \xi_{k} \vert^{2}dx
=  
16\pi 
\left(
\xi_{k}(0)+\ln (W_{k}(0))
\right)
+ O(1).
\end{split}
\end{equation*}
\end{thm}
\begin{proof}
It suffices to observe that the results established above apply to the sequence:
$\hat{\xi}_{k}:=\xi_{k} + \chi_{k}$,  
where $\chi_{k}$ is the unique solution of the 
Dirichlet problem:
\begin{equation*}
\begin{cases}
\Delta \chi_{k} = g_{k} \; \text{ in } \; B_{r} \\
\chi_{k} = 0  \; \text{ in } \;  \partial B_{r}. 
\end{cases}
\end{equation*}
Indeed, by elliptic estimates, $\chi_{k}$ converges strongly 
in $C^{1,\alpha}(B_{r})$ and moreover,
$\ln W_{k}(0)= \sum_{j=1}^{s}2\alpha_{j}\ln \vert p_{j,k} \vert + O(1).$
Then it is easy to check that in terms of the original sequence $\xi_{k}$, 
we get exactly the claimed estimates.  
\end{proof}

\section{Appendix: The Proof of  \eqref{2.33}.}

By virtue of the properties pointed out in Section \ref{sec_preliminaries}
for the sequence $\varphi_{k}$ defined in \eqref{2.12}, 
to establish \eqref{2.33} we can follow word by word the arguments 
used in \cite{Lee_Lin_Yang_Zhang} in order to show the same identity 
(i.e.(4.9) in \cite{Lee_Lin_Yang_Zhang}), see also \cite{Lee_Lin_Tarantello_Yang}.  
To this purpose we start with the following: 
\begin{lemma}
Let $\mu$ be defined in \eqref{2.19}. Then there exists $R_{0}>0$ sufficiently large, such that, as
$k\longrightarrow +\infty$,
\begin{equation}\label{A.1}
\nabla \varphi_{k}(x)
\longrightarrow
-\mu \frac{x}{\vert x \vert^{2}}
+
O(\frac{1}{\vert x \vert^{2}})
\end{equation}
uniformly on compact sets of $\R^{2}\setminus B_{R_{0}}$.
\end{lemma}
\begin{proof}
From \eqref{2.17} we have
\begin{equation*}
\begin{split}
\nabla \varphi_{k}(x)
= &
-
\frac{\tau_{k}}{2\pi}
\int_{B_{r}} (\frac{\tau_{k}x-y}{\vert \tau_{k}x-y \vert^{2}})
W_{k}(y)e^{\xi_{k}(y)}dy \\
= &
- \frac{\tau_{k}}{2\pi}
\int_
{
\{\vert y \vert \leq \tau_{k}\vert x \vert^{2}\}
} (\frac{\tau_{k}x-y}{\vert \tau_{k}x-y \vert^{2}})
W_{k}(y)e^{\xi_{k}(y)}dy \\
& -
\frac{\tau_{k}}{2\pi}
\int_{
\{\tau_{k}\vert x \vert^{2} \leq \vert y \vert \leq r \}
} 
(\frac{\tau_{k}x-y}{\vert \tau_{k}x-y \vert^{2}})
W_{k}(y)e^{\xi_{k}(y)}dy \\
= & :
I_{1,k}(x)
+
I_{2,k}(x).
\end{split}
\end{equation*}
Now, for $R_{0}>1$ (to be fixed later) let 
$\vert x \vert \geq R_{0}$
and
$\tau_{k}\vert x \vert^{2} \leq \vert y \vert \leq r$,
then we have:
\begin{equation*}
(1-\frac{1}{R_{0}})\vert y \vert
\leq
\vert \tau_{k}x-y \vert
\leq
(1+\frac{1}{R_{0}})\vert y \vert
\end{equation*}
and therefore
\begin{equation*}
\vert I_{2,k} \vert
\leq
C \int_
{
\{\tau_{k}\vert x \vert^{2} \leq \vert y \vert \leq r\}
}
\frac{\tau_{k}}{\vert y \vert }e^{\xi_{k}(y)}W_{k}(y)dy
\leq
\frac{C}{\vert x \vert^{2}},
\;
\vert x \vert \geq R_{0} >1.
\end{equation*}
By recalling \eqref{2.15} and \eqref{2.12}, we find:
\begin{equation*}
I_{1,k}(x)
=
-\frac{1}{2\pi}
\int_{\{\vert z \vert \leq \vert x \vert^{2}\}}
\frac{(x-z)}{\vert x-z \vert^{2}}
W_{1,k}(z)e^{\varphi_{k}(z)}dz.
\end{equation*}
Therefore, in case
$\varphi_{k}$ satisfies alternative \eqref{uniform_on_compact_minus_infinity},
then  \eqref{A.1} readily follows.
Indeed, $I_{1,k}\longrightarrow 0$, as $k\longrightarrow +\infty$, 
uniformly on compact sets of $\R^{2}\setminus B_{R_{0}}$
and $\mu=0$ in this case.
While, 
in case $\varphi_{k}$ satisfies alternative \eqref{uniform_on_compact_without},
then for $R_{0}>1$ sufficiently large so that,
$S_{\varphi}\subset \subset B_{R_{0}}$,
we have, as $k\longrightarrow +\infty$:
\begin{equation*}
I_{1,k}(x)
\longrightarrow
-\frac{1}{2\pi}
\sum_{q\in S_{\varphi}} \sigma(q) \frac{(x-q)}{\vert x-q \vert^{2}}
\; \text{ uniformly on compact sets of } \;
\R^{2}\setminus B_{R_{0}},
\end{equation*}
and, in this case (see \eqref{2.25}), we have :
\begin{equation*}
\begin{split}
-\frac{1}{2\pi} \sum_{q\in S_{\varphi}}\sigma(q)\frac{(x-q)}{\vert x-q \vert^{2}}
= &
-
\mu \frac{x}{\vert x \vert^{2}}
+
\frac{1}{2\pi}\sum_{q \in S_{\varphi}}
(
\frac{x}{\vert x \vert^{2}}
-
\frac{(x-q)}{\vert x-q \vert^{2}}
) \\
= &
-
\mu \frac{x}{\vert x \vert^{2}}
+
O(\frac{1}{\vert x \vert^{2}})
,
\; \text{ for } \;
\vert x \vert \geq R_{0}.
\end{split}
\end{equation*}
So also in this case \eqref{A.1} is established.
Finally, in case $\varphi_{k}$ satisfies alternative
\eqref{uniform_on_compact} or \eqref{2.26}, 
then by virtue of \eqref{2.30a} and \eqref{2.30} we find that,
\begin{equation}\label{A.4}
W_{1}(x)e^{\varphi_{0}}
\leq
\frac{C}{\vert x \vert^{2(\rho+1)}}
\; \text{ for } \;
\vert x \vert \geq R_{0}
\end{equation}
with $R_{0}\gg 1$ sufficiently large and 
suitable constants $C>0$ and $ \rho\geq 1$.
In this case, as $k\longrightarrow +\infty$, we find:
\begin{equation*}
I_{1,k}(x)
\longrightarrow
-\mu\frac{x}{\vert x \vert^{2}}
+
R_{1}(x)
,
\; \text{ uniformly on compact sets of } \;
\R^{2}\setminus B_{R_{0}}
\end{equation*}
with
\begin{equation*}
\begin{split}
R_{1}(x)
= &
\frac{1}{2\pi}
\int_{\{\vert y \vert \geq \vert x \vert^{2} \}}
W_{1}(x)e^{\varphi_{0}} \\
& +
\frac{1}{2\pi}
\int_{\{\vert y \vert < \vert x \vert^{2}\}}
(
\frac{x}{\vert x \vert^{2}}
-
\frac{x-y}{\vert x-y \vert^{2}}
)
W_{1}(y)e^{\varphi_{0}(y)}
dy
+
R_{2}(x),
\end{split}
\end{equation*}
where $R_{2}(x)=0$ in case alternative \eqref{uniform_on_compact} holds
(i.e. $S_{\varphi}=\emptyset$)
or
\begin{equation*}
R_{2}(x)
=
\frac{1}{2\pi}
\sum_{q\in S_{\varphi}}
\sigma(q)(\frac{x}{\vert x \vert^{2}}-\frac{x-q}{\vert x-q \vert^{2} })
=
O(\frac{1}{\vert x \vert^{2}})
\end{equation*}
in case alternative \eqref{2.26} holds  (i.e. $S_{\varphi}\neq \emptyset$).
Clearly, in view of \eqref{A.4}, we easily check that,
\begin{equation*}
\int_{\{\vert y \vert \geq \vert x \vert^{2}\}}
W_{1}(y)e^{\varphi_{0}(y)}
\leq
\frac{C}{\vert x \vert^{2}}
,
\; \forall \;
\vert x \vert \geq R_{0}.
\end{equation*}
By recalling the following easy facts:
\begin{equation*}
\begin{split}
&
\; \text{ if } \;
\vert x-y \vert \geq \frac{\vert x \vert }{2}
\; \text{ then } \;
\vert \frac{x}{\vert x \vert^{2}} - \frac{x-y}{\vert x-y \vert^{2}} \vert
\leq 4\frac{\vert y \vert }{\vert x \vert^{2}}
\\ &
\; \text{ if } \;
\vert x-y \vert \leq \frac{\vert x \vert }{2}
\; \text{ then } \;
\vert \frac{x}{\vert x \vert^{2}} - \frac{x-y}{\vert x-y \vert^{2}} \vert
\leq \frac{4}{\vert x-y \vert},
\end{split}
\end{equation*}
and by using \eqref{A.4}, for
$\vert x \vert \geq R_{0}$
we derive:
\begin{equation*}
\begin{split}
\vert &
\int_{\{\vert y \vert \leq \vert x \vert^{2}\}}
(
\frac{x}{\vert x \vert^{2}}
-
\frac{x-y}{\vert x-y \vert^{2}}
)
W_{1,k}(y)e^{\varphi_{0}}
dy
\vert
\\
\leq &
\frac{C}{\vert x \vert^{2}}
\int_{\R^{2}} \vert y \vert W_{1}(y)e^{\varphi_{0}}dy
+
C \int_{ \{ \vert x-y \vert \leq \frac{\vert x \vert }{2} \}}
\frac{1}{\vert x-y \vert }W_{1}(y)e^{\varphi_{0}}dy \\
\leq &
\frac{C}{\vert x \vert^{2}}
+
C \int_{ \{\vert x-y \vert \leq \frac{\vert x \vert }{2} \}}
\frac{1}{\vert x-y \vert}\frac{1}{\vert y \vert ^{2(\rho+1)}}dy  \\
= &
\frac{C}{\vert x \vert^{2}}
+
\frac{C}{\vert x \vert ^{2\rho+1}}
\int_
{
\{
\vert \frac{x}{\vert x \vert }-z \vert
\leq
\frac{1}{2}
\}
}
\frac{1}{\vert \frac{x}{\vert x \vert } - z \vert }
\frac{dz}{\vert z \vert^{2(\rho + 1)}}
\leq
\frac{C}{\vert x \vert^{2}}
\; \text{ for } \; 
\vert x \vert \geq R_{0}\geq 1.
\end{split}
\end{equation*}
By combining the estimates above, 
we obtain \eqref{A.1} also when $\varphi_{k}$ satisfies either 
alternative \eqref{uniform_on_compact} or alternative \eqref{2.26}, 
and the proof is completed.
\end{proof}

\

By using \eqref{2.10} and \eqref{A.1}, we can easily show that, as $k\longrightarrow +\infty$:
\begin{equation*}
r \int_{\partial B_{r}}
(
\vert \partial_{\nu}\xi_{k} \vert
-
\frac{1}{2}\vert \nabla \xi_{k} \vert^{2}
)
d\sigma
\longrightarrow
2\pi \frac{m^{2}}{2}
+
o_{r}(1)
\end{equation*}
with $m$ in \eqref{2.8} and 
$o_{r}(1)\longrightarrow 0 $, as $r\longrightarrow 0$, and
\begin{equation*}
R \int_{\partial B_{R}}
(
\vert \partial_{\nu}\xi_{k} \vert
-
\frac{1}{2}\vert \nabla \xi_{k} \vert^{2}
)
d\sigma
\longrightarrow
2\pi \frac{\mu^{2}}{2}
+
o_{R}(1)
\end{equation*}
with $o_{R}(1)\longrightarrow 0 $,
as $R\longrightarrow + \infty$.
\begin{lemma}
There holds: 
\begin{equation*}
m^{2}-\mu^{2}
=
4(\alpha+1)(m-\mu).
\end{equation*}
\end{lemma}
\begin{proof}
As in \cite{Lee_Lin_Yang_Zhang},\cite{Lee_Lin_Tarantello_Yang},
we use the Pohozaev identity
(see e.g. (5.2.14) in \cite{Tarantello_Book}) 
for $\xi_{k}$ (satisfying \eqref{2.0}) in
$B_{r}\setminus B_{R \tau_{k}}$ and obtain:
\begin{equation}\label{A.10}
\begin{split}
r
\int_{\partial B_{r}}
(
\vert \partial_{\nu}\xi_{k} \vert^{2}
& -
\frac{1}{2}\vert \nabla \xi_{k} \vert^{2}
+
W_{k}e^{\xi_{k}}
)
d\sigma \\
& -
R \int_{\partial B_{R}}
(
\vert \partial_{\nu}\varphi_{k} \vert^{2}
-
\frac{1}{2}
\vert \nabla \varphi_{k} \vert^{2}
+
W_{1,k}e^{\varphi_{k}} )d\sigma
\\
=  &
\int_
{
B_{\frac{r}{\tau_{k}} \setminus B_{R}}
}
2
(
1
+
\sum_{j=1}^{s} \alpha_{j}
\frac{y(y-q_{j,k})}{\vert y-q_{j,k} \vert^{2}}
)
W_{1,k}(y)e^{\varphi_{k}}dy
\\
& +
\int_
{
B_{\frac{r}{\tau_{k}} \setminus B_{R}}
}
(
\tau_{k}y
\cdot
\frac{\nabla h_{k}(\tau_{k}y)}{h_{k}(\tau_{k}y)}
)
W_{1,k}(y)e^{\varphi_{k}(y)}dy
.
\end{split}
\end{equation}
By using \eqref{2.7} and \eqref{2.10}, and the convergence \eqref{A.1} for $\varphi_{k}$,
we easily check that, 
for the left hand side $(LHS)_{k}$
of \eqref{A.10}, we find
\begin{equation*}
(LHS)_{k}
\longrightarrow
2\pi(\frac{m^{2}}{2}-\frac{\mu^{2}}{2})
+
o_{r}(1)+o_{R}(1),
\; \text{ as } \; 
k\longrightarrow +\infty.
\end{equation*}
While for the right hand side
$(RHS)_{k}$ of \eqref{A.10}, we have:
\begin{equation*}
\begin{split}
(RHS)_{k}
= &
2(1+\alpha)
\int_{B_{\frac{r}{\tau_{k}}}\setminus B_{R}}
W_{1,k}(y)e^{\varphi_{k}} \\
& +
\sum_{j=1}^{s}
2 \alpha_{j}
\int_{B_{\frac{r}{\tau_{k}}}\setminus B_{R}}
q_{k,j}
\cdot
\frac{y-q_{k,j}}{\vert y-q_{k,j} \vert^{2}}
W_{1,k}e^{\varphi_{k}}
+
o_{r}(1)
\end{split}
\end{equation*}
and we conclude that, as $k\longrightarrow +\infty$,
\begin{equation*}
(RHS)_{k}
\longrightarrow
4\pi(1+\alpha)(m-\mu)+o_{r}(1)+o_{R}(1).
\end{equation*}
Then we obtain the desired conclusion, by letting
$r\longrightarrow 0^{+}$ and $R\longrightarrow +\infty$.
\end{proof}

\

\textbf{Acknowledgments}

\noindent
We thank Dr. Martin Mayer for his precious assistance and 
useful comments during the preparation of the manuscript.

This work was partially supported by:
MIUR excellence project:   
Department of Mathematics, University of Rome "Tor Vergata"
CUP E83C18000100006,
by "Beyond Borders" research project: 
"Variational approaches to PDE" 
and Fondi di Ricerca Scientifica d'Ateneo, 2021:
Research Project 
"Geometric Analysis in Nonlinear PDEs",
CUP E83C22001810005.

\end{document}